\documentclass[12pt]{amsart}
\usepackage{amsfonts,amsmath,amscd}
\usepackage{epsfig,graphics}
\usepackage{amssymb}
\usepackage{amscd}
\usepackage{hyperref}
\usepackage{bbm}
\usepackage{mathabx}
\usepackage{xcolor}

\newcommand{\BN}{{\mathbb{N}}}
\newcommand{\BR}{{\mathbb{R}}}
\newcommand{\BC}{{\mathbb{C}}}

\newcommand{\gd}{\delta}
\newcommand{\gb}{\beta}

\newcommand{\Ga}{\Gamma} 
\newcommand{\gs}{\sigma} 

\newcommand{\gO}{\Omega}
\newcommand{\go}{\omega}

\newcommand{\gl}{\lambda}

\newcommand{\gth}{\theta}

\newcommand{\cC}{{\mathcal{C}}}
\newcommand{\dd}{{\partial}}
\renewcommand{\hat}{\widehat}

\def\vp{\varphi}
\def\e{\varepsilon}

\newcommand{\ol}[1]{\overline{#1}}

\newcommand{\tprob}{\text{Prob}}

\newcommand{\df}{\stackrel{\text{def}}{=}}

\newcommand{\half}{\frac{1}{2}}

\newtheorem{prop}{Proposition}[section]
\newtheorem*{prop*}{Proposition}
\newtheorem{thm}{Theorem}

\newtheorem*{thm*}{Theorem}
\newtheorem{lem}[prop]{Lemma}
\newtheorem{cor}[prop]{Corollary}
\newtheorem*{cor*}{Corollary}

\theoremstyle{definition}

\newtheorem{defn}[prop]{Definition}
\newtheorem*{defn*}{Definition}
\newtheorem{rem}[prop]{Remark}
\newtheorem*{rem*}{Remark}
\newtheorem{exam}[prop]{Example}

\title[Spectral gaps, critical exponents and representations.]{Spectral gaps, critical exponents and representations of negatively curved groups}

\begin{document}
\maketitle
\centerline{\scriptsize KEVIN BOUCHER}
\centerline{\tiny University of Southampton, UK}
\centerline{\tiny (e-mail: kevin.boucher01@gmail.com)}

\begin{abstract}
In this paper we introduce a notion of Poincar\'e exponent for isometric representations of discrete groups on Hilbert spaces. 
Similarly as growth exponents control the geometry this exponent is shown to control the size of spectral gaps.
Following similar ideas as Patterson and Sullivan in \cite{MR450547} \cite{MR556586} it is used in the case of negatively curved groups to construct weakly contained boundary representations \cite{MR2787597} reflecting the spectral properties of the original representation analogously as complementary series representations in the case of semi-simple Lie groups \cite{MR1792293}.
This is exploited to deduced sharp estimates on spectral invariants. 
A quantitive property (T) \'a la Cowling \cite{MR560837} is also established proving uniform bound on the mixing rate of representations of hyperbolic groups with property (T).
Along the way some properties of boundary representations are discussed. A original characterisation of the positivity of the so-called Knapp-Stein operators and certain fusion rules on the boundary complementary series representations \cite{Boucher:2020ab} (\cite{Boyer:2022aa}) are established.
\end{abstract}

{\tiny \textbf{Key words:} Poincar\'e series, Patterson-Sullivan theory, hyperbolic groups, complementary series, boundary representations, mixing rates, spectral gap. }

\tableofcontents

\section{Introduction}
Given a topological group, $G$, together with a continuous unitary representation on a Hilbert space, $(\pi,\mathcal{H})$, quantitative estimates on the operator norms $\|\pi(\mu)\|_\text{op}$ for Borel probability measures, $\mu$, (see \cite{MR171173} for technical details) is a classic problem in spectral methods involving $G$-representations.
The representation $\pi$ of $G$ is said to have \textit{spectral gap} if there exists a Borel probability measure, $\mu$, with $\ol{\langle\text{supp}(\mu)\rangle}=G$ such that $\|\pi(\mu)\|_\text{op}<1$ (see \cite[Prop. 3.4.]{MR3888695} for other characterisations).
Given a Borel set $A\subset G$ and denoting $\text{Avr}_A(\pi)\df\frac{1}{\text{Vol}(A)}\int_A\pi(g) dg$, 
a central question is to obtain a bound on $\|\text{Avr}_A(\pi)\|_\text{op}$ in terms of the volume of $A$,
in other words understand how the spectral gap evolves when $A$ grows in volume.
This has a number of applications going from homogeneous dynamic \cite{MR3888695} to quantum ergodicity \cite{MR3732880} passing through number theory \cite{MR3286482}, counting problems \cite{MR2573139}, Apollonian circle packing problems \cite{MR2784325} or random walk on groups \cite{MR1743100}.

If $G$ is a semi-simple Lie group it follows from a deep machinery involving classification of irreducible unitary representations together with matrix coefficient estimates that precise quantitative estimates on spectral gaps can be obtained for bi-$K$-invariant balls \cite{MR1637840} where $K\subset G$ stands for a maximal compact subgroup of $G$.

In this present work we address the situation of discrete negatively curved groups in a broad sense by setting the role of boundary complementary series representations within the problem of spectral gap estimates.
Ergodic theorems for measure preserving action of negatively curved groups actions has a long history \cite{MR1923970} \cite{MR1645314} \cite{MR3062901} \cite{MR3358037} \cite{MR3429477} but, as far as the author knows, no quantitive description of the spectral gap have been established even in this particular case.

\subsection{Boundary representations}\hfill\break

The subject of boundary representations of discrete groups popularised by Bader and Muchnik in \cite{MR2787597} can be traced back to the work of Figa-Talamanca and Piccardello on free groups \cite{MR710827}.
Bader and Muchnik investigated the irreducibility of the Koopman representations of negatively curved groups on their visual boundaries endowed with their Patterson-Sullivan measures from a dynamical perspective. 
First established by them in the case of fundamental groups of negatively curved manifolds, it was latter extended to CAT(-1) groups by Boyer \cite{MR3692899}   and Gromov hyperbolic groups by Garncarek \cite{Garncarek:2014aa}. 
Those representations recently inspired different work such as classification of type 1 topological hyperbolic groups and their boundary representations in \cite{MR4626714} and \cite{Glasner:2023aa}, classification of additive representations of free groups in relation with spectral gap problems in \cite{quint} or boundary representation of Mapping class groups \cite{Ma:2022aa}.
Following the analogy between hyperbolic groups and uniform lattices of rank 1 semi-simple Lie groups the author introduced a notion of spherical complementary series representations  in \cite{Boucher:2020ab} (see also \cite{Boyer:2022aa}).
Those representations, when they exist, produce a non-tempered family of non-equivalent, irreducible, unitary representations with intermediate spectral gap properties.
Their existence relates to properties such as the Kazhdan property (T) and its natural obstruction namely the Haagerup property \cite{Boucher:2020ab} \cite{Boucher:2020aa}.
They also showed remarkable approximation properties in relation with the so-called Shalom conjecture \cite{Boucher:2023aa}, \cite{MR2078626}, \cite{Nishikawa:2020aa}, \cite{gruetzner:hal-03963498} and \cite[Conj. 35]{MR3382026} for context.

In the semi-simple case, the relation between
complementary series representations and decay of matrix coefficients  \cite{MR1792293} are fundamental in order to investigate spectral gaps.
Our approach is to extend this relation using complementary boundary representations, as a substitute to the semi-simple ones.

\subsection{A example: homogeneous dynamics of $SL_2(\BR)$}\hfill\break

Let us illustrate our point through the example of $SL_2(\BR)$ viewed as a topological Gromov hyperbolic group.
The reader can refer to \cite{MR855239} for material about $SL_2(\BR)$ representation theory or \cite[Sect. 3.4]{MR3071503} for a gentle summary.

Let $\Gamma\subset SL_2(\BR)$ be a discrete, finitely generated and not virtually solvable subgroup (see \cite{MR3929581} for examples)
and $L^2(\Ga\setminus SL_2(\BR))$ be the quasi-regular representation of $SL_2(\BR)$ on square integrable $\Ga$-automorphic functions on $SL_2(\BR)$.
Assume the critical exponent $\gd_\Ga$ of $\Ga$ with respect to the canonical distance on $SL_2(\BR)$ satisfies $\gd_\Ga>\half$.\\
It follows from Lax and Phillips \cite{MR661875} that:
$$L^2(\Ga\setminus SL_2(\BR))\simeq W_{\gl_0}\oplus \dots\oplus W_{\gl_m}\oplus W_\text{temp.}$$
where $W_{\gl_i}$ is a spherical complementary series of Langland parameter $0<s_i\le 1$ ($s=1$ corresponds to the trivial representation) and $W_\text{temp.}$ the tempered spectrum, i.e a $SL_2(\BR)$-representation with $L^{2+\e}$-integrable matrix coefficients \cite{MR946351}.
Each spherical complementary series representation $W_{\gl_i}$ with $0\le i\le m$ is generated by an unique $SO(2)$-invariant $\gl_i=s_i(1-s_i)$-eigenvector  of the hyperbolic Laplace operator on $L^2(\Ga\setminus SL_2(\BR))^{SO(2)}\simeq L^2(\Ga\setminus\mathbb{H}^{2})$ where  $0\le\gl_i\le \frac{1}{4}$.
Moreover, due to Sullivan's theorem \cite{MR556586}, one has $\gl_0=\gd(1-\gd)$.
In particular only finitely many spherical complementary series appear in $L^2(\Ga\setminus SL_2(\BR))$.\\

\begin{rem}
This decomposition has its generalisation in higher dimension with $\Ga\subset SO(n+1,1)$ discrete, geometrically finite with critical exponent $\gd>\frac{n}{2}$ and Zariski dense (see \cite{MR3336838} and references therein). 
\end{rem}

The above example illustrates the general fact that spherical complementary series control the decay of matrix coefficients (see \cite{MR1792293})
and thus spectral gap of average operators such as
$\frac{1}{\text{Vol}(A)}\int_A\pi(g)dg$ for $A\subset SL_2(\BR)$ and $\pi$ a $SL_2(\BR)$-representation (using spectral transfer \cite{MR1637840} \cite{MR3888695}).

Sullivan's theorem guarantees that $L^2(\Ga\setminus SL_2(\BR))$ contains no complementary series for $\gd<s\le 1$ and therefore a control of the mixing rate for the representation in terms of $\gd$.
This is crucial to deduce the correct normalisation of correlation functions in matrix coefficients equidistribution formulas such as in \cite{MR2057305} or \cite[Thm. 1.4]{MR3336838} and also play a role in the study of boundary representations  \cite{MR2787597}, \cite{MR3692899},\cite{Garncarek:2014aa},\cite{Boucher:2020ab}, \cite{Boyer:2022aa}.

A major difficulty while considering discrete groups is that one cannot rely on  integrability properties of matrix coefficients to approach the spectral gap of a given representation as decay of matrix coefficients and spectral gap turn out to be distinct notions:

\begin{exam}\label{exam:intro:spec:int}
Let $\Ga$ be a discrete group and $\Ga'\lhd\Ga$ an infinite normal subgroup with non-amenable quotient $\Ga/{\Ga'}$. 
The quasi-regular representation $(\gl_{\Ga'},\ell^2(\Ga/{\Ga'}))$ is an example of $\Ga$-representation with spectral gap (due to Kesten \cite{MR112053}) but no matrix coefficient in $\ell^p(\Ga)$ for any $p\neq \infty$ nor even if in $c_0(\Ga)$.
\end{exam}

Following ideas of Patterson and Sullivan \cite{MR450547} \cite{MR556586} on their investigation of the geometry and dynamic related to Kleinian groups, the notion of integrability is substituted with a representation theoretic notion of critical exponent:

\subsection{Statement of results}
\subsubsection{Critical exponent}

Let $\Ga$ be a discrete finitely generated group endowed with a $\Ga$-invariant distance, $d_\Ga$, quasi-isometric to a word distance and $\gd$ its critical exponent defined as 
$\gd\df\limsup_r\frac{1}{r}\log|B_r|$
where $B_r$ stands for a $d_\Ga$-ball of radius $r$ in $\Ga$.

A unitary representation $(\pi,\mathcal{H})$ together with a unit vector $v\in\mathcal{H}$ such that $\ol{\text{Span}\{\pi(g)v:g\in\Ga\}}=\mathcal{H}$ is called \textit{cyclic}.
Two cyclic representations $(\pi,v)$ and $(\pi',v')$ are equivalent if there exists an isometric intertwiner, $U$, between $\pi$ and $\pi'$ such that $U(v)=v'$.
The classes of such representations are in bijective correspondence with positive definite functions on $\Gamma$ (see Subsection \ref{subsec:gns:twist} or \cite{Bekka:2019aa}) through the map $[\pi,v]\leftrightarrow \phi_{\pi,v}(.)\df(\pi(.) v,v)\in\BC^\Ga$.\\
We shall use the identification between positive definite functions on $\Ga$ and cyclic representations without explicit mention throughout the rest.\\
A cyclic representation $[\pi,v]$ is called positive if $(\pi(.) v,v)\ge0$ on $\Ga$ 
and we denote the set of classes of positive cyclic unitary representations $\mathbb{P}_+(\Ga)$.

\begin{exam}[Positive cyclic representations]
\begin{enumerate}
\item The trivial representation associated to the constant function, ${\bf1}_\Ga$, on $\Ga$;
\item Given a subgroup $H\subset \Ga$, $[\gl_H,[\text{Dir}_e]]$ where $\gl_H$ stands for the left quasi-regular representation of $\Ga$ on $\ell^2(\Ga/H)$ and $[\text{Dir}_e]=\text{Dir}_H$ the class of $e\in\Ga$ in $\Ga/H$;
\item Given a unitary cocycle $b:\Ga\rightarrow\mathcal{H}$, that is a map that satisfies $b(gh)=b(g)+\pi(g)b(h)$ for all $g,h\in\Ga$ and some unitary representation $(\pi,\mathcal{H})$ (see \cite[p.76]{MR2415834} for details), then $\phi(g)=e^{-t\|b(g)\|^2}\in\mathbb{P}_+(\Ga)$ for all $t\in\BR_+$;
\item Given measure class preserving action $\Ga\curvearrowright(\gO,m)$ and $u\in L^2_+(\gO)$ with $\|u\|_2=1$, $\phi(g)=(\pi_\text{Koop}(g)u,u)\in\mathbb{P}_+(\Ga)$ where $\pi_\text{Koop}$ stands for the $L^2$ Koopman representation:
$$\pi_\text{Koop}(g)f(\go)=\frac{dg_*m}{dm}(\go)f(g^{-1}\go)$$
for $m$-almost every $\go\in\gO$;
\item The real boundary representations, $\pi_s$, \cite{MR2787597} \cite{Garncarek:2014aa} \cite{Boucher:2020ab} together with the constant vector ${\bf1}_{\dd \Ga}$ also belong to $\mathbb{P}_+(\Ga)$.
\end{enumerate}
\end{exam}

More generally, given a unitary representation $(\pi,\mathcal{H})$ of $\Ga$ and a unit vector $v\in \mathcal{H}$ one can consider the cyclic representation $[\pi\otimes \ol{\pi}|_{W_v},v\otimes \ol{v}]$ where $\ol{\pi}$ stands for the contragredient representation and $W_v=\ol{\text{Span}\{\pi(g)\otimes \ol{\pi(g)}(v\otimes \ol{v}):g\in\Ga\}}\subset \mathcal{H}\otimes \ol{\mathcal{H}}$.
Observe that $[\pi\otimes \ol{\pi}|_{W_v},v\otimes \ol{v}]$ can be seen as a positive representation as $(\pi\otimes \ol{\pi}(v\otimes \ol{v}),v\otimes \ol{v})=|(\pi v,v)|^2$.\\

The critical exponent of a positive cyclic representation $[\pi,v]$ is defined as the abscissa of convergence of the weighted Poincar\'e series:
$$\mathcal{P}([\pi,v];t)=\mathcal{P}((\pi v,v);t)\df\sum_{g\in\Ga} e^{-td_\Ga(g,e)}(\pi(g)v,v)$$
and denoted $\gd[\pi,v]$ (or $\gd(\phi)$ for $\phi=(\pi v,v)$).
Observe that $\gd[\pi,v]\le \gd$.

\begin{rem}
In relation with Brook's amenable cover problem a notion of critical exponent for measure preserving transformation was introduced in \cite{Coulon:2018aa} in order to investigate the spectral properties of  $\ell^2(\Ga/{\Ga'})$ in the particular case where $\gd_\Ga=\gd_{\Ga'}$ and $\Ga$ with some  form of hyperbolicity (see \cite{Coulon:2018aa} for details).
As $[\gl_{\Ga'},[\text{Dir}_e]]$ is a positive cyclic representation one can compare the two notions and show that
$\gd[\gl_{\Ga'},[\text{Dir}_e]]=\gd_{\Ga'}\le \gd(\gl_{\Ga'})$ where $\gd_{\Ga'}$ stands for the critical exponent of ${\Ga'}$ in $(\Ga,d_\Ga)$ and $\gd(\gl_{\Ga'})$ the critical exponent of $\gl_{\Ga'}$ in the sense of \cite{Coulon:2018aa}. 
\end{rem}

If $\gl$ stands for the regular representation of $\Ga$ on $\ell^2(\Ga)$ the following spectral inequality holds:

\begin{thm}\label{thm:general:1}
Let $[\pi,v]$ be a positive cyclic representation.
Then for all $f\in \BC[\Ga]$ one has:
$$\|\pi(f)\|_\text{op}\le e^{\half\gd[\pi,v]r(f)}\|\gl(f)\|_\text{op}$$
where $r(f)$ stands for the smallest $r$ such that $\text{supp}(f)\subset B_r$ the $d_\Ga$-ball of radius $r$ around $e\in \Ga$.
\end{thm}
As observed by Shalom in \cite[Lem. 2.3]{MR1792293}: 
every positive cyclic representation $[\pi,v]$ satisfies $\|\gl(f)\|_\text{op}\le \|\pi(f)\|_\text{op}$ for all $f\in\BR_+[\Ga]$.
It is therefore natural to have a spectral estimate that involves the regular representation.\\

A function $f\in \BC[\Ga]$ on $\Ga$ is radial if $f(g)=k(d_\Ga(g,e))$ for some compactly supported function $k$ on $\BR$.
\begin{defn*}[\cite{MR3666050}]
The discrete metric group $(\Ga,d_\Ga)$ has the \textit{radial rapid decay (RRD) property}, if there exist constants $C\in\BR_+$ and $m\in\BN$ such that for all radial function $f\in \BC[\Ga]$ 
supported in a ball of radius $R$: $\|\gl(f)\|_\text{op}\le CR^m\|f\|_2$.
\end{defn*}
This property is satisfied by Gromov hyperbolic groups \cite{MR943303}, CAT(0) cubical groups \cite{MR2153902}, cocompact lattices of connected semi-simple Lie groups \cite{MR2514049} and others  \cite{MR3666050}.
It also appears to be strictly weaker than the usual rapid decay \cite{MR4613611} popularised by Lafforgue in \cite{MR1914617}.

Let $B_r$ be the ball of radius $r$ around the unit element $e\in\Ga$ and denote $\text{Avr}_r(\pi)\df\frac{1}{|B_r|}\sum_{g\in B_r}\pi(g)$.
The asymptotic behaviour of $\|\text{Avr}_r(\pi)\|_\text{op}$ is independent of the centre of $B_r$ and  the \textit{entropy} (or \textit{mixing rate}) of $\pi$ is defined as:
 $$h(\pi)\df-\limsup_r\frac{1}{r}\log \|\text{Avr}_r(\pi)\|_\text{op}\ge0$$

\begin{rem}
Let $\text{Rep}(\Ga)$ be the space of separable unitary $\Ga$-representations.
The map $\pi\in\text{Rep}(\Ga)\mapsto h(\pi)\in\BR_+$ is monotone with respect to weak containment \cite[Appendix F]{MR2415834}. More precisely if $\pi,\pi'\in \text{Rep}(\Ga)$ with $\pi\prec\pi'$ then $h(\pi')\le h(\pi)$ with equality whenever $\pi$ and $\pi'$ are Fell equivalent \cite[Thm. F.4.4]{MR2415834}.
\end{rem}

The following result establishes a general relation between the entropy and critical exponent:
\begin{cor}
If $\Ga$ has the radial rapid decay property, there exists $m\in\BN$, for all $\e>0$, there exist $C(\e),R_\e>0$ such that
$$\|\text{Avr}_r(\pi)\|_\text{op}\le C(\e) r^m e^{\half (\gd(\pi)-\gd-\e) r}$$
for all $r\ge R_\e$.
In particular $h(\pi)\ge \half (\gd-\gd[\pi,v])$.
\end{cor}

\begin{cor*}
If $\Ga$ has the radial rapid decay property and $\Ga'\subset \Ga$ a subgroup with $\gd_{\Ga'}<\gd_\Ga$,
then mixing rate of $\ell^2(\Ga/{\Ga'})$ is strictly positive.
\end{cor*}

A unitary representation, $(\pi,\mathcal{H})$, is called strongly $L^{p+\e}$ with $p\neq \infty$ if for all $p'>p$, there exists a dense subspace, $\mathcal{D}_{p'}\subset \mathcal{H}$, such that the matrix coefficients $(\pi(.) v, w)\in \ell^{p'}(\Ga)$ for all $v,w\in \mathcal{D}_{p'}$.

\begin{thm}\label{thm:general:2}
Let $\pi$ be a strongly $L^{p+\e}$ unitary representation with $2\le p\neq \infty$, then for all $f\in\BC[\Ga]$:
$$\|\pi(f)\|_\text{op}\le e^{\frac{p-2}{2p}\gd r(f)}\|\gl(f)\|_\text{op}.$$
\end{thm}

In particular if $(\Ga,d_\Ga)$ has the radial rapid decay property one has $h(\pi)\ge \frac{1}{p}\gd$.

\begin{rem}
It follows from \cite[Thm. 1]{MR1637840} and \cite[Lem. 2.3]{MR1792293} that 
$\|\pi(f)\|_\text{op}= \|\gl(f)\|_\text{op}$ for all $f\in\BR_+[\Ga]$ whenever $[\pi,v]$ is positive and strongly $L^{p+\e}$ with $1\le p\le 2$.
\end{rem}

\begin{rem}
It was known by spectral transfer argument (see \cite[Thm. 1]{MR1637840}) that $\|\text{Avr}_r(\pi)\|_\text{op}\le Ce^{-\gb r}$ for some $\gb(p)>0$ whenever $\pi$ is strongly $L^p$ and $(\Ga,d_\Ga)$ has the radial rapid decay. As far as the author knows such precise bound is new in this general context.
\end{rem}

Observe that Theorem \ref{thm:general:2} shows some deficiency in Theorem \ref{thm:general:1}.
Indeed assuming the Poincar\'e series $t\mapsto\sum_{g\in\Ga}e^{-td_\Ga(g,e)}$ diverges at $\gd$, i.e $\Ga$ is of divergent type,
the functions $f_s(g)=e^{-(1-s)\gd d_\Ga(g,e)}$ for $\half< s< 1$ belong to $\ell^{\frac{1}{1-s}+\e}(\Ga)\setminus \ell^{\frac{1}{1-s}}(\Ga)$ with $p=\frac{1}{1-s}>2$ and satisfy 
$\gd(f_s)=s\gd$.
If $f_s$ (or $\phi_s\asymp f_s$ with $\phi_s\in\mathbb{P}_+(\Ga)$) is positive definite, Theorem \ref{thm:general:2} gives a sharper estimate of the mixing rate than Theorem \ref{thm:general:1} (see Subsection \ref{subsec:gap:mixing:general} for more details).
This suggest the better estimate on mixing rate in terms of the critical exponent : $h(\pi)\ge \gd-\gd[\pi,v]$.

As discussed in the next subsection this sharper bound is actually optimal under the assumption that $\Ga$ is hyperbolic.

\subsubsection{Critical exponent and negatively curved groups}
\label{intro:general:mixing}\hfill\break

Assume for the rest of this introduction that $\Ga$ is a non-elementary Gromov hyperbolic group and $d_\Ga$ is strongly hyperbolic (see Definition \ref{def:str:hyp:spa.nica}).
\begin{rem}
Due to Mineyev and Yu in \cite{MR1914618} (see also \cite{MR2346214} and \cite{MR3551185}) every Gromov hyperbolic group acts geometrically on a proper roughly geodesic strongly hyperbolic space.
\end{rem}

\begin{thm}\label{thm:intro:entropy:lim}
Let $[\pi,v]$ be a positive cyclic representation with critical exponent $\gd[\pi,v]$.
Then 
$$h(\pi)=-\lim_r\frac{1}{r}\log \|\text{Avr}_r(\pi)\|_\text{op}=\gd-\max\{\gd[\pi,v],\half\gd\}$$
In particular $0\le h(\pi)\le \frac{\gd}{2}$ and $\gd[\pi,v]$ is independent of the choice of the positive cyclic vector $v$ and can be denoted $\gd(\pi)$.
\end{thm}

A careful analysis of the critical case $\gd=\gd(\pi)$ leads to:

\begin{thm}\label{thm:intro:charact}
Let $[\pi,v]$ be a positive cyclic representation with critical exponent $\gd(\pi)$.
The following are equivalent:
\begin{enumerate}
\item  $\pi$ has almost invariant vectors;
\item  $\gd(\pi)=\gd$, i.e $[\pi,v]$ has no critical gap;
\item  $\pi$ has trivial entropy, i.e $h(\pi)= 0$.
\end{enumerate}
\end{thm}
This result can be seen as a generalisation of \cite[Thm. 1]{Coulon:2018aa} with $[\pi,v]=[\gl_{{\Ga'}},[\text{Dir}_e]]$.\\

A representation, $(\pi,\mathcal{H})$, is called weakly mixing if it does not contain any finite dimension subrepresentation or equivalently $(\pi\otimes \ol{\pi},\mathcal{H}\otimes \ol{\mathcal{H}})$ does not have non-zero invariant vectors \cite[Prop A.1.12]{MR2415834}.
Similarly as Cowling's quantitative property (T) \cite{MR560837} for semi-simple Lie groups, hyperbolic groups with Kazhdan property (T) \cite{MR2415834} have a uniform lower bound on the entropy of their mixing representations:

\begin{thm}\label{thm:quant:T}
If $\Ga$ has the property (T) then there exists $\gb(\Ga,d_\Ga)>0$, for all mixing representation, $\pi$, one has
$h(\pi)\ge \gb(\Ga,d_\Ga)$.
\end{thm}

The main ingredient in our approach is the existence of a critical weakly contained boundary representation of conformal order $\gd(\pi)$ :

\subsubsection{Critical exponent and boundary representations}$ $

Let $\dd \Ga$ be the Gromov boundary of $\Ga$, the Gromov product based at $e\in \Ga$ between $\xi,\eta\in\dd \Ga$ is denoted $(\xi,\eta)\in[0,\infty]$ and $b_x$ the Buseman cocycle at $x\in\Ga\cup\dd\Ga$ (see Subsection \ref{subsec:prelim:strong}). 
Let $\pi_s$ be the linear representation of $\Ga$ on $\cC(\ol{\Ga})$ given by
$$\pi_s(g)\psi(x)=e^{-s\gd b_x(g,e)}\psi(g^{-1}x)$$
for all $g\in\Ga$ and $x\in\ol{\Ga}$.

The analogue of spherical complementary series are determined by the (bounded) $s$-Knapp-Stein operators, $\mathcal{I}_s$, for $\half<s<1$  defined on $L^2(\nu_\text{PS})$ as:
$$\mathcal{I}_s(\psi)(\xi)=\int_{\dd \Ga}\psi(\eta)e^{2(1-s)\gd(\xi,\eta)}d\nu_\text{PS}(\eta)$$
where $\nu_\text{PS}$ stands for the Patterson-Sullivan measure.
The $s$-spherical complementary series (also called natural complementary series in the rest) are defined whenever $\mathcal{I}_s$ is positive definite and are obtained from the completion on $L^2(\nu_\text{PS})$ with respect to the $\pi_s$-invariant quadratic form:
$$Q_s(\psi)=(\mathcal{I}_s(\psi),\psi)_{L^2(\nu_\text{PS})}$$
see \cite{Boucher:2020ab} (\cite{Boyer:2022aa}) for more details and \cite[Chap. 5]{MR710827} for the case of free groups.

As non virtually abelian discrete groups are not type $1$ \cite[Chap. 7]{Bekka:2019aa} their unitary dual cannot be canonically understood.
This is reflected on the boundary through the existence of boundary representation of exotic type that does not have analogue within the semi-simple case and motivate the following definition:

Given $\nu$ a $\Ga$-quasi-invariant probability measure over the boundary $\dd \Ga$, a conditional expectation, $E$, on $L^1(\nu)$ is called \textit{$(\Ga,s,\nu)$-conformal} if it satisfies $\pi^*_s(g)E=E\pi_s(g)$ for all $g\in\Ga$ on $L^2(\nu)\subset L^1(\nu)$, where $\pi^*_s(g)$ stands for the adjoint of $\pi_s(g)$ on $L^2(\nu)$.
Assuming the quadratic form
$$Q_E(f)\df\iint f(\xi)E(f)(\xi) d\nu(\xi)$$
associated to $E$ on $L^2(\nu)$ is positive definite it induces a $\pi_s$-invariant scalar product and therefore a $\Ga$-representation. 
A boundary representation is called $s$-complementary series representation (or $(\nu,E,s)$-complementary series to insist on the structure) if it is obtained through the above procedure.
\begin{thm}\label{thm:main:ingre:1}
Let $(\Ga,d_\Ga)$ be Gromov hyperbolic endowed with a strongly hyperbolic distance as above and $[\pi,v]$ be a positive cyclic representation with critical exponent $\gd([\pi,v])$.
There exists a $(\nu,E,s)$-complementary series representation, $\pi^\dd$, with $s\gd=\gd[\pi,v]$ factorising through the kernel of $\pi$ that is weakly contained in $\pi$.
Moreover $\gd(\pi^\dd)=s\gd =\gd[\pi,v]$ with $\pi^\dd=1$ whenever $\gd[\pi,v]=\gd$ and $\pi$ does not weakly contain any boundary representation of conformal index strictly greater than $\gd(\pi)$.
\end{thm}

In certain situations the representations $\pi^\dd$ induce by Theorem \ref{thm:main:ingre:1} can be realised as an actual subrepresentation of $\pi$ or even be identified with the natural complementary series induced by the Knapp-Stein operator:

\begin{defn}
A positive function $f$ on $\Ga$ is called Poincar\'e spherical if there exist $\gd(f)<s_0$ and $C$ such that the inequality:
 $$\sum f(g^{-1}h)e^{-\gs d_\Ga(h,e)}e^{-\gs d_\Ga(g,e)}\le C.\mathcal{P}(f;\gs)^2$$ holds for all $\gd(f)<\gs\le s_0$.
\end{defn}
See Subsection \ref{subsec:weak:cont:andco} for more explanations about this definition.

\begin{prop}\label{prop:intro:poin:sph}
Let $\phi=(\pi v,v)\in\mathbb{P}_+(\Ga)$ that is Poincar\'e spherical and $[\pi,v]$ its positive cyclic representation.
Then the space of intertwiners between $\pi^\dd$ and $\pi$ is non-trivial.
Moreover if $\pi^\dd$ is irreducible then $\pi^\dd$ is a subrepresentation of $\pi$.
\end{prop}

\begin{rem}
It is likely that any boundary representation obtained from the above procedure must be irreducible. 
\end{rem}

Given a real positive function, $k$, a function $f$ on $\Ga$ is called $k$-roughly-spherical if $f\asymp k(d_\Ga(\bullet,e))$.
Let $k_s$ be the real function $k_s(t)=\exp({-(1-s)\gd t})$ for $t\ge 0$.
Then a $k_s$-roughly-spherical function on $\Ga$ is Poincar\'e spherical (see Proposition \ref{prop:hs:rough}) and
the spectral properties of those type of functions relate to the spectrum of the Knapp-Stein operators introduced before:
\begin{thm}\label{thm:intro:knapp:char}
If a pointwise positive, positive definite function $\phi$ is $k_s$-roughly radial then the Knapp-Stein operator $\mathcal{I}_s$ is positive definite on $L^2(\nu_\text{PS})$, the boundary representation $\pi_\phi^\dd$ is uniquely defined and isomorphic to the natural complementary series $\pi_s$.
Moreover $\pi_s$ appears as a subrepresentation of $\pi_\phi$.
\end{thm}

This leads to the following characterisation of positivity for the Knapp-Stein operator $\mathcal{I}_s$ on $L^2(\nu_\text{PS})$:
\begin{cor}\label{cor:intro:charact:matrix}
Let $k_s$ be the real function $k_s(t)=\exp({-(1-s)\gd t})$.
The operator $\mathcal{I}_s$ on $L^2(\nu_\text{PS})$ is positive definite if and only if $\Ga$ has a $k_s$-roughly-spherical positive definite function.
\end{cor}

We call a unitary cocycle $b:\Ga\rightarrow \mathcal{H}$ \textit{special} if the Hilbert pseudo-distance $d_b(g,h)=\|b(g^{-1}h)\|^2$ for $g,h\in\Ga$ is roughly isometric to $d_\Ga$, i.e $\|d_\Ga- c.d_b\|_\infty\le C$ for some finite $C$ and $c>0$.
\begin{cor}[Fusion rules]\label{cor:intro:fusion:special}
Let $\half<s,s',t\le 1$ such that $s+s'=1+t$ and assuming $\pi_s$ and $\pi_{s'}$ are well defined  natural complementary series, then $\pi_t$ is also well defined and $\pi_t\subset \pi_s\otimes \pi_{s'}$.
Moreover if $\Ga$ admits a special cocycle $c$ then the natural complementary series $(\pi_s)_{s\in(\half,1]}$ exist and $\pi_s\subset e^{-(1-s)\|c\|^2}$ for all $\half<s\le 1$.
\end{cor}


\subsection{Organisation of the paper}\hfill\break

Section \ref{sec:prelim} starts with general material about negatively curved spaces and groups. 
Estimates on spherical sums of the $s$-density representations are discussed in Subsection \ref{subsec:fund:esti}.
Elements of Banach geometry and their relations with the so-called Gelfand-Naimark-Segal (or GNS) construction of positive cyclic representations are respectively discussed in Subsections \ref{subsec:net:pos} and \ref{subsec:gns:twist}.

The section \ref{sec:const:bound} is dedicated to the constructive part of this work, namely the boundary representations of positive cyclic representations.
After discussing properties of Poincare series on negatively curved groups in Subsection \ref{subsec:crit:poin}, the construction is performed throughout the Subsections \ref{subsec:twist:poincare} to \ref{subsec:quad:final}.

Our main results are proved in Sections \ref{sec:first:prop} and \ref{sec:spec:bound:all}.
In Section \ref{sec:first:prop} we prove general properties about the boundary representations constructed in Section \ref{sec:const:bound} and give specific description of them in particular cases. The Section \ref{sec:spec:bound:all} focuses on spectral estimates.

Theorems \ref{thm:general:1} and \ref{thm:general:2} are respectively established in the self-contained Subsection \ref{subsec:gap:mixing:general}.
Theorems \ref{thm:intro:entropy:lim} and \ref{thm:intro:charact}  are respectively consequences of Corollary \ref{cor:pos:rep:spectrapgap} and \ref{cor:gap:eq:almost} in Section \ref{subsec:last}.
Subsection \ref{subsec:quant:prop:T} is dedicated to the proof of Theorem \ref{thm:quant:T}.
Theorem \ref{thm:main:ingre:1} is a consequence of the construction of Section \ref{sec:const:bound}, Subsection \ref{subsec:weak:cont:andco} and Corollary \ref{cor:last:of:the:last}.
Theorem \ref{thm:intro:knapp:char} together with Corollary \ref{cor:intro:charact:matrix} and \ref{cor:intro:fusion:special} are proved in subsection \ref{subsec:nat:comp}.
The Proposition \ref{prop:intro:poin:sph} is established in Subsection \ref{subsec:weak:cont:andco}.

\subsection{Notations and symbols}\hfill\break
Given two real functions $a,b:S\rightarrow\BR$ on a set $S$, we denote $a\prec_* b$ if there exists a constant $C$ that depends on $*$ such that $a(s)\le Cb(s)$ for all $s\in S$ and $a\asymp_* b$ if $a\prec_* b$ and $b\prec_* a$.\\
Given two real functions $a,b:S\rightarrow\BR$ on a set $S$, we denote $a\lesssim_* b$ if there exists a constant $C$ that depends on $*$ such that $a(s)\le b(s)+C$ for all $s\in S$ and $a\simeq_* b$ if $a\lesssim_* b$ and $b\lesssim_* a$.\\

\subsection{Acknowledgments}\hfill\break
The authors would like to express his gratitude to Uri Bader, Alex Furman and Lam Pham for inspiring discussions that led to the present work as well as Bogdan Nica and Jan Spakula for their support and feedbacks and Rhiannon Dougall for interesting discussions about boundary representations.

The author was supported by EPSRC Standard Grant EP/V002899/1.

\section{Preliminaries}\label{sec:prelim}
In this section we recall basics about Gromov hyperbolic geometry and counting arguments.
The reader can refer to \cite{MR1086648} for more about Gromov hyperbolic geometry .

Let $(X,d_X)$ be a discrete uniformly locally finite metric space and $o\in X$ a fixed basepoint. 

\subsection{Strong hyperbolicity and consequences}\label{subsec:prelim:strong}
\hfill \break
A metric space $(X,d_X)$ is \textit{roughly geodesic} if for all $x,y\in X$, there exist a interval $I\subset\BR$ and $c:I\rightarrow X$ such that
$d_X(c(t_1),c(t_2))\simeq_X|t_1-t_2|$ for all $t_1,t_2\in I$.

The \textit{Gromov product} on $(X,d_X)$ at $o\in X$ is defined as
$$(x,y)_o\df\half(d_X(x,o)+d_X(y,o)-d_X(x,y))$$
for all $x,y\in X$ and $(X,d_X)$ is \textit{Gromov hyperbolic} if there exists $C\ge0$ such that
$$(x,y)_o\ge \min\{(x,z)_o,(z,y)_o\}-C$$
for all $x,y,z\in X$.

Assuming $(X,d_X)$ is Gromov hyperbolic the \textit{Gromov boundary} of $(X,d_X)$ can be defined as 
$$\dd X\df\{(x_n)_n\in X^\BN:\liminf_{n,m}(x_n,x_m)_o=+\infty\}/\sim$$
with $(x_n)_n\sim(y_n)_n\Leftrightarrow\liminf_{n,m}(x_n,y_m)_o=+\infty$.
The space $\dd X$ endowed with the natural visual topology is compact as the space $(X,d_X)$ is proper. Moreover there exists a compatible topology on $\ol{X}=X\cup\dd X$ that makes it a compact space \cite{MR1086648}. 

\begin{defn}[\cite{MR3551185}]\label{def:str:hyp:spa.nica}
A roughly geodesic hyperbolic metric space $(X,d_X)$ is strongly hyperbolic if there exists $\e>0$ such that
$$e^{-\e (x,y)_o}\le e^{-\e (x,z)_o}+e^{-\e (z,y)_o}$$
for all $x,y,z,o\in X$.
\end{defn}

Assuming $(X,d_X)$ is strongly hyperbolic, the Gromov product extends continuously to the Gromov compactification $\ol{X}=X\cup \dd X$ (excluding the diagonal pairs in $\dd X$, i.e $(\xi,\xi)$ with $\xi\in\dd X$) and a visual distance $d_{\dd X}$ on $\dd X$ can be defined by $d_{\dd X}(\xi,\eta)=e^{-\e_0 (\xi,\eta)_o}$ for all $\xi,\eta\in \dd X$ where $\e_0>0$ is chosen small enough.
In particular $\lim_{z\rightarrow \xi} d_X(x,z)-d_X(y,z)=b_\xi(x,y)=d_X(x,y)-2(x,\xi)_y$ for all $\xi\in\dd X$ and $x,y\in X$.

\begin{thm*}[Mineyev-Yu \cite{MR1914618} ]
Every Gromov hyperbolic group admits an invariant roughly geodesic and strongly hyperbolic distance on themselves that is quasi-isometric to a word distance.
\end{thm*}
See \cite{MR2346214} and \cite{MR3551185} for more about the subject.

\subsection{Orbital distribution}\label{subsec:prelim:orbital}\hfill\break

Given $x\in X$ we denote $\hat{x}_o\subset \dd X$ the non-empty compact subset defined as 
$$\hat{x}_o=\{\xi\in\dd X:(x,\xi)_o=\|(x,\bullet)_o\|_\infty\}$$
As a consequence of strong hyperbolicity the map $\xi\mapsto (\xi,x)_o$ for $x\in X$ is continuous on the compact space $\dd X$.

\begin{rem}
The \textit{hat map}, $x\mapsto \hat{x}_o$, can be formalised using the notion of correspondences \cite[Chap. 17]{MR2378491}.
\end{rem}

Using a tree approximation argument one can show that 
$$(x,\xi)_o\simeq d_X(x,o)$$
for all $x\in X$ and $\xi\in \hat{x}_o$.
Moreover, using strong hyperbolicity,
$d_{\dd X}(\xi,\eta)=e^{-\e(\xi,\eta)_o}\prec e^{-\e d_X(x,o)}$ for all $\xi,\eta\in \hat{x}_o$, in other words $\text{diam}(\hat{x}_o)\prec e^{-\e d_X(x,o)}$.
In the rest we use the notation $\hat{x}$ for any $\xi\in\hat{x}_o$ and relation that does not involve a particular choice.

There exist $C_1,C_2>0$ and $L>R>0$ such that
$$B(\hat{x},C_1e^{-\e_0 d_X(x,o)})\subset \mathcal{O}_o(x;R)\subset B(\hat{x},C_2e^{-\e_0 d_X(x,o)})$$
where $\mathcal{O}_o(x;R)\df \{\xi\in\dd X:(\xi,x)_o\ge d(x,o)-R \}$ stands for $R$-shadow at $x\in X$ with $d_X(o,x)\ge L$ (see \cite[Prop. 2.1]{MR2919980}).
A sufficient condition for the above $R>0$ is that any rough geodesic from $o$ to $\xi\in \dd X$ intersects 
$$C_{o,m}\df \{x\in X: Rm\le d_X(x,o)< R(m+1)\}$$
and $\dd X=\bigcup_{x\in C_{o,m}}\mathcal{O}_o(x;R)$ for all $m\in\BN$ .
The reader can refer to \cite[Sect. 4]{Garncarek:2014aa} for elements of proof (see also \cite{MR2787597} \cite{MR3692899}).

Observe that 
$\mathcal{O}_o(x;R)\cap \mathcal{O}_o(y;R)\neq \emptyset$ implies that 
$d_X(x,y)\le 4R$ and as $(X,d_X)$ is uniformly locally finite the covering
$\mathcal{C}_{o,m}=\{\mathcal{O}_o(x;R)\}_{x\in C_{o,m}}$ have uniformly finite multiplicity independently of $m$, i.e there exists $M\ge1$, for all $\xi\in\dd X$, $|\{x\in C_{o,m}: \xi\in\mathcal{O}(x;R)\}|\le M$.

\begin{defn}
Given for $D>0$, a metric measure space $(Z,d,\nu)$ is called $D$-David-Ahlfors regular (or $D$-regular) if
$\nu(B(\xi,r))\asymp r^D$ for all $r\in [0,\text{diam}(Z,d)]$.
\end{defn}
Equivalently \cite{MR1616732} given a ball of radius $0<r<\text{diam}(Z,d)$, any maximal $\e$-discrete sequence in $B(\xi,r)$ has a number of elements equivalent, up to multiplicative constants independent of $\xi\in Z$ and $r$, to $\frac{r^D}{\e^D}$.

We conclude this subsection with Vitali's Lemma \cite{MR1800917} and some consequences:
\begin{lem}\label{lem:vita}[Vitali]
Given an arbitrary collection, $\mathcal{C}$, of balls in a separable metric space $(Z,d)$ of uniformly bounded radius.
There exists a countable subcollection $\mathcal{C}'$ of pairwise disjoint balls such that
$$\bigcup_\mathcal{C}B=\bigcup_{\mathcal{C}'}5B'$$
\end{lem}

Since the cover $(\mathcal{O}_o(x;R))_{x\in C_{o,m}}$ has uniformly finite multiplicity, the Vitali's covering Lemma \ref{lem:vita} implies:
\begin{cor}\label{cor:vita:shadow}
If $(\dd X,d_{\dd X})$ is $D$-regular with respect to his Hausdorff measure (or any other measure \cite{MR1800917}), then for all $\xi\in \dd X$ and $0<r<\text{diam}(\dd X,d_{\dd X})$ the number of elements in $\mathcal{C}_{o,m}$ that intersect $B(\xi,r)$ is equivalent, up to multiplicative constants independent of $\xi\in Z$ and $r$, to $e^{\e_0D Rm}r^D$.
\end{cor}

\begin{rem}
In the case $(X,d_X)=(\Ga,d_\Ga)$ considered throughout the rest of this work, one has $\e_0\text{Dim}_\text{Haus}(\dd \Ga,d_{\dd \Ga})=\gd_\Ga$, where $\gd_\Ga$ is the critical exponent of $(\Ga,d_\Ga)$ \cite{MR1214072}.
\end{rem}

\begin{defn}
A discrete metric space $(Y,d_Y)$ is called purely spherical at $y\in Y$ if there exists $R$, $|B_Y(y,r)|\asymp e^{\gd_Y r}$ for $r\ge R$ where $\gd_Y$ for the critical exponent of $(Y,d_Y)$.
\end{defn}

\begin{lem}
The space $(X,d_X)$ is purely spherical at $o\in X$ whenever $\dd X$ is David-Ahlfors regular.
\end{lem}

\begin{proof}
Let $C_{o,m}^*\subset C_{o,m}$ maximal such that $\dd X= \bigcup_{x\in C_{o,m}^*}B(\hat{x},5C_2e^{-\e_0 d_X(x,o)})$ and $(B(\hat{x},C_2e^{-\e_0 d_X(x,o)}))_{x\in C_{o,m}^*}$ pairwise disjoint.
Then using $D$-regularity $|C_{o,m}^*|e^{-\e_0D Rm}\asymp 1$.

On the other hand it follows from maximality that for all $x\in C_{o,m}$, there exists $x'\in C_{o,m}^*$ such that $B(\hat{x},C_2e^{-\e_0 d_X(x,o)})\cap B(\hat{x}',C_2e^{-\e_0 d_X(x,o)})\neq\emptyset$ and uniform multiplicity that the number of $x$ that intersect a given ball $B(\hat{x}',C_2e^{-\e_0 d_X(x,o)})$ is uniformly controlled.
In particular $|C_{o,m}^*|\asymp |C_{o,m}|$.
\end{proof}

Let 
$$\mathcal{O}_{o}(m,n;y)\df \{x\in C_{o,m}:Rn\le (x,y)_o<R(n+1)\}$$
$$\mathcal{O}_{o}'(m,n;y)\df \{x\in C_{o,m}:\exists \xi\in\hat{x}, Rn\le (\xi,y)_o<R(n+1)\}$$
for $y\in X\cup\dd X$, $m,n\in\BN$ with $n\le m$ and $C\ge e^{\e_0 R}$.
It follows from Corollary \ref{cor:vita:shadow} that
\begin{equation}\label{eq:fund:esti}
\frac{|\mathcal{O}_{o}'(m,n;y)|}{|C_{o,m}|}\asymp\frac{|\mathcal{O}_{o}(m,n;y)|}{|C_{o,m}|}\asymp e^{-\gd Rn}
\end{equation}
whenever $\dd X$ is $D$-regular (see also \cite[Lem. 4.3]{Garncarek:2014aa} for alternative proof).

\subsection{$s$-density representations and summation estimates}\label{subsec:fund:esti}\hfill\break
In the rest $(X,d_X)=(\Ga,d_\Ga)$, $o=e$ and we denote $C_{m}=C_{o,m}$. In particular $\dd \Ga$ is $D$-regular with respect to his visual distance due to the Patterson-Sullivan theory and $\e_0D_{(\e_0)}=\gd$ (see \cite[Thm. 7.7]{MR1214072}).

\begin{defn}\label{def:slowgrowth}
A slow growing function is a non-decreasing function $\gth:\BR_+\rightarrow[1,+\infty)$ such that for all $\e>0$, there exists $T$, for all $u\ge0$ and $t\ge T$: $\gth(u+t)\le e^{\e u}\gth(t)$.
\end{defn}
Note that $\gth=1$ is a slow growing function.

Given $\gth$ and $s\in[0,1]$, we define the $\gth$-rearranged  $s$-density representation, $\pi_{\gth,s}$, on $\cC(\ol{\Ga})$ as 
$$\pi_{\gth,s}(g)\phi(x)=e^{-s\gd b_x(g)}\frac{\gth(d_\Ga(x,g))}{\gth(d_\Ga(x,e))}\phi(g^{-1}x)$$
for all $x\in \ol{\Ga}$.
Observe that the continuous extension over the boundary, $\dd \Ga$, satisfies $\pi_{\gth,s}|_{\dd\Ga}=\pi_s|_{\dd\Ga}=\pi_{1,s}|_{\dd\Ga}$.

The following estimates are extensively used in the rest:
\begin{lem}[Version 1]\label{lem:tech:v1}
Let $\gs_2\in (\half,1]$, $\gs_1\in\BR_+$, $\pi_{\gs_2}$ the $\gs_2$-density representation and

$$\mathcal{S}_m(\gs_1,\gs_2;h)\df (\sum_{g\in C_m} e^{-\gs_1\gd d_\Ga(g,e)}\pi_{\gs_2}(g)){\bf1}_h$$

Then
\[   
\mathcal{S}_m(\gs_1,\gs_2;h)\asymp_{\gs_2}
     \begin{cases}
       e^{(1-(\gs_1+\gs_2))\gd Rm}e^{(2\gs_2-1)\gd\min\{Rm,d_\Ga(h,e)\}} & \text{ for $h\in\Ga$}\\
       e^{(\gs_2-\gs_1)\gd Rm} &\text{otherwise} 
     \end{cases}
\]

\end{lem}

\begin{proof}

For $h\in\Ga$:
\begin{align*}
\mathcal{S}_m(\gs_1,\gs_2;h)&\df (\sum_{g\in C_m} e^{-\gs_1\gd d_\Ga(g,e)}\pi_{\gs_2}(g)){\bf1}_h\\
&=\sum_{g\in C_m} e^{-\gs_1 \gd d_\Ga(g,e)}e^{-\gs_2\gd b_h(g^{-1},e)}\\
&=\sum_{g\in C_m}e^{2\gs_2\gd (g,h)_e}e^{-(\gs_1+\gs_2)\gd d_\Ga(g,e)}\\
&\asymp e^{(1-(\gs_1+\gs_2))\gd Rm}[\frac{1}{|C_m|}\sum_{C_m}e^{2\gs_2\gd (g,h)_e}]
\end{align*}

Using the estimate (\ref{eq:fund:esti}): 
\begin{align*}
&\frac{1}{|C_m|}\sum_{C_m}e^{2\gs_2\gd(g,h)_e}
\asymp \sum_{n:Rn\le \min\{Rm,d_\Ga(h,e)\}}\frac{|\mathcal{O}(m,n;h)|}{|C_m|}e^{2\gs_2\gd Rn}\\
&\asymp \sum_{n:Rn\le \min\{Rm,d_\Ga(h,e)\}}e^{(2\gs_2-1)\gd Rn}
\asymp_{\gs_2} e^{(2\gs_2-1)\gd\min\{Rm,d_\Ga(h,e)\}}
\end{align*}

and thus
\begin{align*}
\mathcal{S}_m(\gs_1,\gs_2;h)
&\asymp_{\gs_2} e^{(1-(\gs_1+\gs_2))\gd Rm}e^{(2\gs_2-1)\gd\min\{Rm,d_\Ga(h,e)\}}
\end{align*}
The case $h=\xi\in\dd \Ga$ follows from similar computations replacing $\min\{Rm,d_\Ga(h,e)\}$ with $Rm$.

\end{proof}

\begin{lem}[Version 2]\label{lem:tech:v2}
Let $\gth$ be a slow growing function, $\gs_2\in (\half,1]$, $\gs_1\in\BR_+$, $\pi_{\gth,\gs_2}$ the rearranged $\gs_2$-density representation and
$$\mathcal{S}_{\gth,m}(\gs_1,\gs_2;h)\df (\sum_{g\in C_m} \gth(d_\Ga(g,e))e^{-\gs_1\gd d_\Ga(g,e)}\pi_{\gs_2,\gth}(g)){\bf1}_h$$
for $\gs_2>\half$.
Then for all $\e>0$, there exists $N\ge0$ such that:
\[   
\mathcal{S}_{\gth,m}(\gs_1,\gs_2;h)\prec_{\gs_2} 
     \begin{cases}
       \gth^2(Rm)e^{\e d_\Ga(h,e)}e^{(1-(\gs_1+\gs_2))\gd Rm}e^{(2\gs_2-1)\gd\min\{Rm,d_\Ga(h,e)\}} & \text{ for $h\in\Ga$}\\
       \gth^2(Rm)e^{\e d_\Ga(h,e)}e^{(\gs_2-\gs_1)\gd Rm} &\text{otherwise} 
     \end{cases}
\]
for all $m\ge N$.

\end{lem}

\begin{proof}
First observe that if $\gth$ has slow growth then $\gth^2$ as well and $\gth\le \gth^2$ as $\gth\ge1$.
Let 
$$k_\gth(g,h)\df \frac{\gth(d_\Ga(h,g))}{\gth(d_\Ga(h,e))\gth(d_\Ga(g,e))}\le\frac{\gth(d_\Ga(g,e)+d_\Ga(h,e))}{\gth(d_\Ga(h,e))\gth(d_\Ga(g,e))} $$
for $g,h\in\Ga$.
Given $\e>0$ and $T_\e$ as in Definition \ref{def:slowgrowth}, 
for $d_\Ga(g,e)\ge T_\e$ one has:
$$k_\gth(g,h)\le\frac{e^{\e d_\Ga(h,e)}}{\gth(d_\Ga(h,e))}$$
for all $h\in\Ga$.

Given $h\in\Ga$:
\begin{align*}
\mathcal{S}_{\gth,m}(\gs_1,\gs_2;h)&\df (\sum_{g\in C_m} \gth(d_\Ga(g,e))e^{-\gs_1\gd d_\Ga(g,e)}\pi_{\gth,\gs_2}(g)){\bf1}_h\\
&=\sum_{g\in C_m} \gth(d_\Ga(g,e))e^{-\gs_1 \gd d_\Ga(g,e)}\frac{\gth(d_\Ga(h,g))}{\gth(d_\Ga(h,e))}e^{-\gs_2\gd b_h(g^{-1},e)}\\
&=\sum_{g\in C_m}k_\gth(g,h)e^{2\gs_2\gd (g,h)_e}\gth^2(d_\Ga(g,e))e^{-(\gs_1+\gs_2)\gd d_\Ga(g,e)}\\
&\asymp\gth^2(Rm)e^{(1-(\gs_1+\gs_2))\gd Rm}[\frac{1}{|C_m|}\sum_{g\in C_m}k_\gth(g,h)e^{2\gs_2\gd (g,h)_e}]
\end{align*}

For $m\ge T_\e$, doing a similar computation as in Lemma \ref{lem:tech:v1} proof:
\begin{align*}
&\frac{1}{|C_m|}\sum_{g\in C_m}k_\gth(g,h)e^{2\gs_2(g,h)_e}
\prec \frac{e^{\e d_\Ga(h,e)}}{\gth(d_\Ga(h,e))}\frac{1}{|C_m|}\sum_{g\in C_m}e^{2\gs_2(g,h)_e}\\
&\prec_{\gs_2} e^{\e d_\Ga(h,e)}e^{(2\gs_2-1)\gd\min\{Rm,d_\Ga(h,e)\}}
\end{align*}

and thus:
\begin{align*}
\mathcal{S}_{\gth,m}(\gs_1,\gs_2;h)
&\prec_{\gs_2}\gth^2(Rm)e^{\e d_\Ga(h,e)}e^{(1-(\gs_1+\gs_2))\gd Rm}e^{(2\gs_2-1)\gd\min\{Rm,d_\Ga(h,e)\}}
\end{align*}
The case $h=\xi\in\dd \Ga$ follows from 
similar computations replacing $\min\{Rm,d_\Ga(h,e)\}$ with $Rm$
and using the fact that $\frac{\gth(d_\Ga(h,g))}{\gth(d_\Ga(h,e))}\xrightarrow{d_\Ga(h,e)\rightarrow \infty} 1$.

\end{proof}

\begin{rem}\label{rem:tech:half}
Following the same lines as Lemma \ref{lem:tech:v1}:\\
Given
$$\mathcal{S}_m(\gs,\half;h)\df (\sum_{g\in C_m} e^{-\gs\gd d_\Ga(g,e)}\pi_{\half}(g)){\bf1}_h$$
with $\gs\in\BR_+$,
there exist $N>0$ and $L>0$ such that
\[   
\mathcal{S}_m(\gs,\half;h)\asymp
     \begin{cases}
       e^{(\half-\gs)\gd Rm}\min\{Rm,d_\Ga(h,e)\} & \text{ for $h\in\Ga$}\\
       e^{(\half-\gs)\gd Rm}Rm &\text{otherwise} 
     \end{cases}
\]
for $m\ge N$, $d_\Ga(h,e)\ge L$.

\end{rem}

\begin{cor}\label{cor:tech:sum:last}
Let $\gs_2\in [\half,1]$, $\gs_1\in\BR_+$ with $\gs_1>\gs_2$, $\pi_{\gs_2}$ the $\gs_2$-density representation and 
$$\mathcal{S}_{\infty}(\gs_1,\gs_2;\xi)=(\sum_{g\in\Ga} e^{-\gs_1\gd d_\Ga(g,e)}\pi_{\gs_2}(g)){\bf1}_\xi$$
for $\xi\in\dd \Ga$.
Then
\[   
\mathcal{S}_\infty(\gs_1,\gs_2;\xi)\asymp_{\gs_2}
     \begin{cases}
       \frac{1}{\gs_1-\gs_2}& \text{for $\gs_2>\half$}\\
       \frac{1}{(\gs_1-\half)^2} &\text{for $\gs_2=\half$} 
     \end{cases}
\]
In particular $\mathcal{S}_\infty(\gs_1,\gs_2;\bullet)$ is bounded on $\dd \Ga$.
\end{cor}

\begin{proof}
Given $\xi\in\dd \Ga$ and $\gs_2>\half$, one has:
\begin{align*}
\mathcal{S}_{\infty}(\gs_1,\gs_2;\xi)&=(\sum_{g\in\Ga} e^{-\gs_1\gd d_\Ga(g,e)}\pi_{\gs_2}(g)){\bf1}_\xi
=\sum_{m}\mathcal{S}_{m}(\gs_1,\gs_2;\xi)\\
&\asymp_\text{Lem. \ref{lem:tech:v1}} \sum_me^{(\gs_2-\gs_1)\gd Rm}\\
&\asymp \frac{1}{\gs_1-\gs_2}
\end{align*}

The case $\gs_2=\half$ follows from similar computations using remark \ref{rem:tech:half} and a summation by part argument.
\end{proof}

\begin{cor}\label{cor:ball:spec:tech}
Let $s\in [\half,1]$ and $\pi_s$ the $s$-density representation.
Then there exists $L_0$, for all $L\ge L_0$:
\[   
\sum_{g:d_\Ga(g,e)\le L}\pi_s(g){\bf1}|_{\dd \Ga}\asymp_s
     \begin{cases}
       e^{s\gd L}& \text{for $s>\half$}\\
       e^{s\gd L}L&\text{for $s=\half$} 
     \end{cases}
\]
Moreover 
\[   
\sum_{g:d_\Ga(g,e)\le L}\pi_s(g){\bf1}|_{\ol{\Ga}}\prec_s
     \begin{cases}
       e^{s\gd L}& \text{for $s>\half$}\\
       e^{s\gd L}L&\text{for $s=\half$} 
     \end{cases}
\]

\end{cor}
\begin{proof}
Given $\xi\in\dd \Ga$ and $\gs_2>\half$, one has:
\begin{align*}
\sum_{g:d_\Ga(g,e)\le L}\pi_s(g){\bf1}_\xi&=\sum_{m}\mathcal{S}_{m}(0,s;\xi)
\asymp_{s,\text{Lem. \ref{lem:tech:v1}}}\sum_{m:mR\le L}e^{s\gd Rm}\\
&\asymp_s e^{s\gd L}
\end{align*}
The case $\gs_2=\half$ follows from similar computations using remark \ref{rem:tech:half} and summation by part.
The upper bound for $x\in \Ga\subset\ol{\Ga}$ follows from the fact that $\min\{Rm,d_\Ga(h,e)\}\le Rm$ and the above computations.
\end{proof}

\subsection{Nets and positive families of vectors}\label{subsec:net:pos}\hfill\break
The reader can refer to \cite[Chap. 2]{MR2378491} for details about the notion of net and order.

Given a Hilbert space, $\mathcal{H}$, and countable set, $I$,
a family of vectors $\{v_i\}_{i\in I}$ \textit{belongs to a positive cone} in $\mathcal{H}$ if $(v_i,v_j)\ge0$ for all $i,j\in I$.

Given such a family, $\{v_i\}_{i\in I}$, we define the summation net as
$$(2^I_0,\subset)\rightarrow(\mathcal{H},\mathcal{H}_+); F\subset I\mapsto v_F=\sum_Fv_i$$
where $(2^I_0,\subset)$ stands for the finite subsets of $I$ with direct order induced by the inclusion. Observe that $(\|v_F\|)_F$ is a non-decreasing net
as  $$\|v+w\|^2=\|v\|^2+\|w\|^2+2(v,w)\ge\max\{\|v\|^2, \|w\|^2\}$$ for all $v,w\in \mathcal{C}_+$
with $\mathcal{C}_+\df\ol{\{\sum_Fc_iv_i:\text{$c_i\ge0$ for all $i\in F$}\}}\subset\mathcal{H}$.

\begin{lem}\label{lem:sum:net}
Given a family $\{v_i\}_{i\in I}$ within a positive cone of $\mathcal{H}$, the summation net $(v_F)_F$ converges if and only if $(\|v_F\|)_F$ is bounded.
\end{lem}

In particular $\sum_{i\in I}v_i$ is well defined iff there exists an exhaustion $(F_n)_n$ by finite sets of $I$ such that $\sup_n\|\sum_{F_n}v_i\|$ is finite.

\begin{proof}
It is enough to prove that if $(\|v_F\|)_F$ is bounded the net $(v_F)_F$ converges in $\mathcal{H}$.

As $(\|v_F\|)_F$ is monotone and bounded it converges to $\sup_{2^I_0}\|v_F\|$.
In particular  $(\|v_F\|)_F$ is a Cauchy net.\\
On the other hand given $F_0\subset F$ one has:
$$\|v_{F\setminus F_0}\|^2=\|v_F-v_{F_0}\|^2=\|v_F\|^2-\|v_{F_0}\|^2-2(v_{F_0},v_{F\setminus F_0})\le\|v_{F}\|^2-\|v_{F_0}\|^2$$
that is small whenever $F_0$ is taken large enough and as 
$$\|v_{F}-v_{F'}\|\le \|v_{F'\setminus F_0}\|+\|v_{F\setminus F_0}\|$$
for $F_0\subset F',F$ it follows that $(v_F)_F$ is a Cauchy net in the Banach space $\mathcal{H}$ and thus converges to a unique limit point.
\end{proof}

\begin{lem}\label{lem:infty:net}
Given a summable family $\{v_i\}_{i\in I}$ within a positive cone of $\mathcal{H}$ and $u=(u_i)_{i\in I},u'=(u'_i)_{i\in I}\in \ell^\infty_+(I)$ with $u\le u'$.
Then $$\|\sum_{i\in I} u_iv_i\|\le\|\sum_{i\in I} u_i'v_i\|\le \|u'\|_\infty\|\sum_{i\in I} v_i\|$$
\end{lem}
\begin{proof}
It enough to prove the first inequality on a finite subset, $F$, of $I$. The result will follow from Lemma \ref{lem:sum:net}.

In this case 
\begin{align*}
\|\sum_{i\in F} u_iv_i\|^2&=\sum_{i,j\in F} u_iu_j(v_i,v_j)
\le\sum_{i,j\in F} u_i'u_j'(v_i,v_j)=\|\sum_{i\in F} u_i'v_i\|^2\\
\end{align*}
\end{proof}

\subsection{Observation about the GNS of pointwise positive functions}\label{subsec:gns:twist}\hfill\break
A function $\phi$ is called positive definite if the quadratic form on $\BC[\Ga]$ given by 
$$B_\phi(f)=\sum_{g,h}f(g)\ol{f(h)}\phi(g^{-1}h)$$
is positive \cite[Sect. 1.B]{Bekka:2019aa}. As it is invariant by left translation it defines a pre-unitary cyclic representation with $\text{Dir}_e$ as cyclic vector.
Moreover one has $\phi(g)=B_\phi(\gl(g)\text{Dir}_e,\text{Dir}_e)$ for all $g\in\Ga$.
This construction, called Gelfand-Naimak-Segal construction (or GNS), defines a bijection between positive definite functions and classes of cyclic representations of $\Ga$.

Let $\mathbb{P}_+(\Ga)$ be the set of positive definite functions on $\Ga$ that are pointwise positive. Given $\phi\in \mathbb{P}_+(\Ga)$ we denote $(\pi_\phi,\mathcal{H}_\phi,{\bf1}_\phi)$ its GNS triple.

Observation that $(\pi_\phi(g){\bf1}_\phi)_g$ belongs to a positive cones of $\mathcal{H}_\phi$ as defined in Subsection \ref{subsec:net:pos}.
This follows from the fact that $$(\pi_\phi(g){\bf1}_\phi,\pi_\phi(h){\bf1}_\phi)=(\pi_\phi(h^{-1}g){\bf1}_\phi,{\bf1}_\phi)=\phi(h^{-1}g)\ge0$$
Given a weight $\go$ with $\go(g)>0$ for all $g\in\Ga$ such that $\sum_g\go(g)\pi_\phi(g){\bf1}_\phi$ is well defined, Lemma \ref{lem:infty:net} implies that $\sum_g\go(g)|f|(g)\pi_\phi(g){\bf1}_\phi$ is also well defined with 
$$\|\sum_g\go(g)|f|(g)\pi_\phi(g){\bf1}_\phi\|\le \|f\|_\infty\|\sum_g\go(g)\pi_\phi(g){\bf1}_\phi\|$$
for all $f\in \ell^\infty(\Ga)$.
One can construct a pre-Hilbertian structure, $B_{\phi,\go}$, on $\ell^\infty(\Ga)$ as follows: 
\begin{align*}
\mathcal{P}_\go^{2}.B_{\phi,\go}(f)&\df 
\sum_{g,h\in\Ga} f(g)f(h)\phi(g^{-1}h)\go(g)\go(h)\\
&=_\text{Lem. \ref{lem:sum:net}}\|\sum_g\go(g)f(g)\pi_\phi(g){\bf1}_\phi\|^2\\
&=\|\sum_g\go(g)f_+(g)\pi_\phi(g){\bf1}_\phi-\sum_g\go(g)f_-(g)\pi_\phi(g){\bf1}_\phi\|^2\\
&\le 2\|\sum_g\go(g)f_+(g)\pi_\phi(g){\bf1}_\phi\|^2+2\|\sum_g\go(g)f_-(g)\pi_\phi(g){\bf1}_\phi\|^2\\
&\le_\text{Lem. \ref{lem:infty:net}} 4\|f\|_\infty^2\|\sum_g\go(g)\pi_\phi(g){\bf1}_\phi\|^2
\end{align*}
with $\|\sum_g\go(g)\pi_\phi(g){\bf1}_\phi\|^2=\mathcal{P}_\go^{2}$ and $f\in \ell^\infty(\Ga)$.
The case of complex-valued functions is obtained by scalar extension.

For this positive quadratic form the linear operators 
 $$\pi_\go(g)f(z)=\frac{\go(g^{-1}z)}{\go(z)}f(g^{-1}z)$$
 for $g,z\in\Ga$ defines a pre-unitary representation. 

Observe that the operator
$$J_\go:\ell^\infty(\Ga)\rightarrow \mathcal{H}_\phi;\quad f\mapsto \frac{1}{\mathcal{P}_\go}\sum_z \go(z)f(z)\pi_\phi(z){\bf 1}_\phi$$
defines a bounded $(\pi_\go,\pi)$-intertwiner with dense range.
Indeed given $f'\in\BC[\Ga]$ one has $\pi_\phi(f'){\bf1}=J_\go(f)$ with $f(g)=\mathcal{P}_\go\go(g)^{-1}f'(g)$.

The operator $J_\go$ naturally extends to a Hilbert space of functions:

Observe that
\begin{align*}
&B_{\phi,\go}(f_1,f_2)=\mathcal{P}_\go^{-2}\sum_{x,y\in\Ga} f_1(x)f_2(y)\go(x)\go(y)\phi(x^{-1}y)\\
&=\sum_{x\in\Ga} f_1(x)\frac{\sum_{y\in\Ga}f_2(y)\go(y)\phi(x^{-1}y)}{\sum_{y\in\Ga}\go(y)\phi(x^{-1}y)}\mathcal{P}_\go^{-2}.\go(x)\sum_{y\in\Ga}\go(y)\phi(x^{-1}y)\\
&=\int_{\Ga} f_1(x) E_\go(f_2)_xd\mu_\go(x)
\end{align*}
where $\mu_\go$ denote a probability on $\ol{\Ga}$ defined as:
$$\int_\Ga f(x)d\mu_\go(x)=\mathcal{P}_\go^{-2}\int_\Ga f(x)\go(x)\sum_{y\in\Ga}\go(y)\phi(x^{-1}y)$$
and $E_\go$ the $(\ol{\Ga},\mu_\go)$-conditional expectation (\cite[Chap. \romannumeral 2]{MR488194}):
$$E_\go(f)_x=\frac{\sum_{y\in\Ga}f(y)\go(y)\phi(x^{-1}y)}{\sum_{y\in\Ga}\go(y)\phi(x^{-1}y)}$$
for $f\in L^1(\mu_\go)$.

In particular $E_\go$ induces a self-adjoint contraction on $L^2(\mu_\go)$ with $E_\go({\bf1})={\bf1}$ and thus $B_{\phi,\go}(f)=\|J_\go(f)|\mathcal{H}_\phi\|^2\le \|f|L^2(\mu_\go)\|^2$ for $f\in \ell^\infty(\Ga)$. 
It follows that $J_\go$ extends to $L^2(\mu_\go)$.

Moreover
$$E_\go(\pi_\go(g)f)_x=\frac{\sum_{y\in\Ga}\frac{\go(g^{-1}y)}{\go(y)}\go(y)\phi(x^{-1}y)}{\sum_{y\in\Ga}\go(y)\phi(x^{-1}y)}E_\go(f)_{g^{-1}x }$$
in other words
 $E_\go(\pi_\go(g)f)=\pi_\go^*(g)E_\go(f)$ 
 where $f\in L^2(\mu_\go)$ and $\pi_\go^*$ denote the adjoint of $\pi_\go$ in $L^2(\mu_\go)$.
In particular $\text{Ker}J_\go$ is $\pi_\go$-invariant and $J_\go|_N$ with $N=L^2(\mu_\go)/{\text{Ker}J_\go}\simeq\text{Ker}J_\go^\perp$ extends as a unitary equivalence between the completion of $N$ with respect to the quadratic form $B_{\phi,\go}$, $\ol{(N,B_{\phi,\go})}$, and $(\pi_\phi,\mathcal{H}_\phi)$ and satisfies 
$$[J_\go|_{\ol{N}}]^{-1}\pi_\phi J_\go|_{\ol{N}}=\pi_{\go}$$


\section{Boundary representations of positive cyclic representations}\label{sec:const:bound}
This section is dedicated to the construction of boundary representations associated to positive cyclic representations. 
Those representations are obtained as limit of twisted GNS construction (see Subsection \ref{subsec:gns:twist}).

\subsection{Critical exponent and Poincar\'e series}\label{subsec:crit:poin}\hfill\break
Given a pointwise positive function on $\Ga$, $f$, the normalised critical exponent of $f$ is defined as the infimum of $\gs>0$ such that:
$$\mathcal{P}(f;\gs)=\sum_{g\in\Ga} e^{-\gs \gd d_\Ga(g,e)}f(g)$$
is finite and denoted $s(f)$.
The critical exponent of $f$ is $\gd(f)\df s(f)\gd$.
A pointwise positive function is of \textit{divergent type} if $\mathcal{P}(f;\gs)\xrightarrow{\gs\rightarrow s(f)^+}+\infty$.

\begin{rem}
The indicator function of a subgroup of $\Ga$ is of divergent type if and only if the subgroup is of divergent type \cite{MR450547}.
\end{rem}

As in the previous constructions (e.g \cite{MR1214072} \cite{Coulon:2018aa} \cite{MR450547})  the divergence of the Poincar\'e series $\mathcal{P}(f;\gs)$  at $\gd(f)$ is necessary.
This can be arranged by an adaptation of the so-called Patterson argument \cite{MR450547}.
The following sequential formulation of Patterson's trick is used later to deal with families of positive cyclic representations.
\begin{lem}\label{lem:patt:arg}
Let $(f_n)_n$ be a sequence of non-zero pointwise positive functions on $\Ga$ with $\lim_n\gd(f_n)=\gd_\infty$ finite.
Let $(\gs_n)_n$ be a non-increasing sequence such that $\gs_n\gd>\gd(f_n)$ and $\gs_n\gd\rightarrow \gd_\infty$. 
There exists a non-decreasing map $\gth:\BR_+\rightarrow[1,+\infty)$ such that:
\begin{enumerate}
\item for all $\e>0$, there exists $T$, for all $u\ge0$ and $t\ge T$: $\gth(u+t)\le e^{\e u}\gth(t)$, i.e $\gth$ has slow growth;
\item $$\mathcal{P}_\gth(f_n;\gs)\df \sum_{g\in\Ga}f_n(g)\gth(d_\Ga(g,e))e^{-\gs\gd d_\Ga(g,e)}$$ is bounded when $\gd(f_n)<\gs\gd$ and unbounded when $\gd(f_n)>\gs\gd$;
\item $\lim_n\mathcal{P}_\gth(f_n;\gs_n)\rightarrow +\infty$.
\end{enumerate}

\end{lem}

\begin{proof}
Let $(\e_n)_n$ be a non-increasing sequence such that $\e_n>0$, $\lim_n\e_n=0$ and $(\gs_n-\e_n)\gd<\gd(f_n)$.

We construct the function $\gth$ together with a non-decreasing positive sequence $(t_n)_n$ inductively as follows. 
Set $t_0=0$ and $\gth(t_0)=1$ and define
$t_{n+1}$ such that $t_{n+1}>t_n+1$ and
$$e^{-\e_nt_n}\gth(t_n)\sum_{\{g:t_n<d_\Ga(g,e)\le t_{n+1}\}}e^{-(\gs_n-\e_n)\gd d_\Ga(g,e)}f_n(g)\ge n$$ 
and $\gth$ on $]t_n,t_{n+1}]$ to be $\gth(t)=e^{\e_n(t-t_n)}\gth(t_n)$.

Observe that
\begin{align*}
\mathcal{P}_\gth(f_n;\gs_n)&\ge\sum_{\{g:t_n<d_\Ga(g,e)\le t_{n+1}\}}\gth(d_\Ga(g,e))e^{-\gs_n\gd d_\Ga(g,e)}f_n(g)\\
&=e^{-\e_nt_n}\gth(t_n)\sum_{\{g:t_n<d_\Ga(g,e)\le t_{n+1}\}}e^{-(\gs_n-\e_n)\gd d_\Ga(g,e)}f_n(g)\ge n
\end{align*}

The rest follows from the definition of $\gth$.

\end{proof}

\begin{prop}\label{prop:two:div}
Let $f\neq0$ be a positive function on $\Ga$ with $\gd(f)$ finite and $\gth$ as in Lemma \ref{lem:patt:arg}.
Then
$$\mathcal{P}_{\gth,2}(f;\gs)\df \sum_{g,h\in\Ga} f(g^{-1}h)\gth(d_\Ga(g,e))e^{-\gs \gd d_\Ga(g,e)}\gth(d_\Ga(h,e))e^{-\gs \gd d_\Ga(h,e)}$$
is bounded  for $\max\{\half\gd,\gd(f)\}<\gs\gd$, unbounded for $\max\{\half\gd,\gd(f)\}>\gs\gd$ and 
$$\mathcal{P}_{\gth,2}(f;\gs)\xrightarrow{\gs\rightarrow \max\{\half\gd,\gd(f)\}^+}+\infty$$
\end{prop}

\begin{proof}
Let $\gamma\in \{f>0\}$ and $\gs>0$ such that $\mathcal{P}_{\gth,2}(f;\gs)$ is finite.
Then
\begin{align*}
f(\gamma)&\sum_{g\in\Ga} (\gth(d_\Ga(g,e)))^2e^{-2\gs\gd d_\Ga(g,e)}\\
&\asymp_{\gth,\gamma} f(\gamma)\sum_{g\in\Ga} \gth(d_\Ga(g\gamma,e))e^{-\gs\gd d_\Ga(g\gamma,e)}\gth(d_\Ga(g,e))e^{-\gs\gd d_\Ga(g,e)}\\
&\le\mathcal{P}_{\gth,2}(f;\gs)
\end{align*}
and thus we must have $\gs>\half$.
On the other hand
$$\mathcal{P}_{\gth}(f;\gs)=\sum_\Ga e^{-\gs\gd d_\Ga(g,e)}\gth(d_\Ga(g,e))f(g)\le \mathcal{P}_{\gth,2}(f;\gs)$$ 
and thus 
$\mathcal{P}_{\gth,2}(f;\gs)\xrightarrow{\gs\rightarrow \max\{\half\gd,\gd(f)\}^+}+\infty$
as $\Ga$ is of divergence type \cite[Cor 7.3]{MR1214072}.

Let us prove that $\mathcal{P}_{\gth,2}(f;\gs)$ is finite whenever $\gs\gd>\max\{\gd(f),\half\gd\}$.
Assume $\gd(f)>\half\gd$ (the case $\gd(f)\le\half\gd$ follows from a similar approach together with remark \ref{rem:tech:half}).
Let $\gs$ such that $\gs\gd>\gd(f)$, $\e$ such that $(\gs-3\e)\gd>\gd(f)$ and $T$ such that $\gth(u+t)\le e^{\gd\e u}\gth(t)$ for all $u\ge0$ and $t\ge T$ as in Lemma \ref{lem:patt:arg}.
Using Lemma \ref{lem:tech:v2} and its notation:
\begin{align*}
\mathcal{P}_{\gth,2}&(f;\gs)
=\sum_{\Ga\times\Ga} f(g^{-1}h)\gth(d_\Ga(g,e))e^{-\gs \gd d_\Ga(g,e)}\gth(d_\Ga(h,e))e^{-\gs \gd d_\Ga(h,e)}\\
&=_{h'=g^{-1}h}\sum_{h'\in\Ga} f(h')\gth(d_\Ga(h',e))e^{-\gs \gd d_\Ga(h',e)}[\sum_{g\in\Ga} \gth(d_\Ga(g,e))e^{-\gs \gd d_\Ga(g,e)}\pi_{\gth,\gs}(g){\bf1}_{h'}]\\
&=\sum_{h\in\Ga} f(h)\gth(d_\Ga(h,e))e^{-\gs \gd d_\Ga(h,e)}(\sum_{m:Rm\le T}+\sum_{m:T< Rm})\mathcal{S}_{\gth,m}(\gs,\gs;h)\\
&\le\|\sum_{m:Rm\le T}\mathcal{S}_{\gth,m}(\gs,\gs;\bullet)\|_\infty\mathcal{P}_{\gth}(f,\gs)\\
&+\sum_{h\in\Ga} f(h)\gth(d_\Ga(h,e))e^{-\gs \gd d_\Ga(h,e)}\sum_{m:T< Rm}\mathcal{S}_{\gth,m}(\gs,\gs;h)
\end{align*}

It enough to prove that the second term, $I$, is finite.

\begin{align*}
I&=\sum_{h\in\Ga} f(h)\gth(d_\Ga(h,e))e^{-\gs\gd d_\Ga(h,e)}(\sum_{m:T< Rm\le d_\Ga(h,e)}+\sum_{m:\max\{T,d_\Ga(h,e)\}< Rm})\mathcal{S}_{\gth,m}(\gs,\gs;h)\\
&\prec\sum_{h\in\Ga} f(h)\gth(d_\Ga(h,e))e^{-(\gs-\e) \gd d_\Ga(h,e)}\sum_{m:T< Rm\le d_\Ga(h,e)}e^{2\e\gd Rm}\\
&+\sum_{h\in\Ga} f(h)\gth(d_\Ga(h,e))e^{(\gs-1-\e) \gd d_\Ga(h,e)}\sum_{m:\max\{T,d_\Ga(h,e)\}< Rm}e^{(1+2\e-2\gs)\gd Rm}\\
&\prec\frac{1}{e^{2\e\gd R}-1}\sum_{h\in\Ga} f(h)\gth(d_\Ga(h,e))e^{-(\gs-3\e) \gd d_\Ga(h,e)}\\
&+\frac{1}{1-e^{(1+2\e-2\gs)\gd R}}
\sum_{h\in\Ga} f(h)\gth(d_\Ga(h,e))e^{-(\gs-\e)\gd d_\Ga(h,e)}\\
&=\frac{1}{e^{2\e\gd R}-1}\mathcal{P}_{\gth}(f,\gs-3\e)+\frac{1}{1-e^{(1+2\e-2\gs)\gd R}}\mathcal{P}_{\gth}(f,\gs-3\e)
\end{align*}

that is finite.
\end{proof}

\begin{rem}
Note that $\mathcal{P}_{\gth=1,2}(\phi)$ diverges at $\half\gd$ whenever $\gd(f)\le\half\gd$, in other words no rearrangement is needed.
\end{rem}

\begin{rem}
Given a sequence of pointwise positive functions on $\Ga$, $(f_n)_n$, with  $\lim_n\gd(f_n)=\gd_\infty$ finite, $(\gs_n)_n$ a non-increasing sequence such that $\gs_n\gd>\gd(f_n)$ and $\gs_n\gd\rightarrow \gd_\infty$ and $\gth$ has in Lemma \label{lem:patt:arg}.
It follows from Proposition \ref{prop:two:div} that the sequence is well defined
$$\mathcal{P}_{\gth,2}(f_n;\gs_n)\df \sum_{g,h\in\Ga} f_n(g^{-1}h)\gth(d_\Ga(g,e))e^{-\gs_n \gd d_\Ga(g,e)}\gth(d_\Ga(h,e))e^{-\gs_n \gd d_\Ga(h,e)}$$
and 
$$\mathcal{P}_{\gth,2}(f_n;\gs_n)\xrightarrow{\gs\rightarrow \max\{\half\gd,\gd_\infty\}^+}+\infty$$
\end{rem}

\subsection{A revisited Gelfand-Naimark-Segal construction}\label{subsec:twist:poincare}\hfill\break

Let ${\boldsymbol\phi}=(\phi_n\simeq[\pi_{\phi_n},\mathcal{H}_{\phi_n},{\bf1}_{\phi_n}])_n$ in $\mathbb{P}_+(\Ga)$ such that $\lim_n\gd(\phi_n)=\gd({\boldsymbol\phi})\le \gd$ is well defined.
Let $(\gs_n)_n$ be a non-increasing sequence such that $\gs_n\gd>\gd(\phi_n)$ and $\gs_n\gd\rightarrow \gd({\boldsymbol\phi})$ and $\gth$ as in Lemma \ref{lem:patt:arg}. 
As a consequence of Lemma \ref{lem:sum:net} the sum 
$$\sum_g\gth(d_\Ga(g,e))e^{-\gs_n\gd d_\Ga(g,e)}\pi_{\phi_n}(g){\bf1}_{\phi_n}$$ 
is well defined in $\mathcal{H}_{\phi_n}$ if and only if
$$\sup_m\sum_{g,h\in F_m} \gth(d_\Ga(g,e))e^{-\gs_n\gd d_\Ga(g,e)}\gth(d_\Ga(h,e))e^{-\gs_n\gd d_\Ga(h,e)}\phi_n(g^{-1}h)$$
is finite for some exhaustion, $(F_m)_m$, of $\Ga$ by finite sets which is equivalent in the case $\gd_\infty\ge\half\gd$, 
using Proposition \ref{prop:two:div}, to
$$\mathcal{P}_{\gth}(\phi_n;\gs_n)=\sum_{g\in\Ga} \gth(d_\Ga(g,e))e^{-\gs_n\gd d_\Ga(g,e)}\phi_{n}(g)$$ is finite.

It follows from Subsection \ref{subsec:gns:twist} that the maps:
$$J_{n}:(\cC(\ol{\Ga}),\|\bullet\|_\infty)\rightarrow \mathcal{H}_{\phi_n};\quad f\mapsto \frac{1}{\sqrt{\mathcal{P}_{\gth,2}(\phi_n;\gs_n)}}\sum_z e^{-\gs_n\gd d_\Ga(z,e)}\gth(d_\Ga(z,e))f(z)\pi_{\phi_n}(z){\bf 1}_{\phi_n}$$
for $n\ge0$, are well defined bounded intertwiners between the representations 
$$\pi_{\gth,\gs_n}(g)f(z)=\frac{\gth(d_\Ga(g,z))}{\gth(d_\Ga(z,e))}e^{-\gs_n\gd b_z(g)}f(g^{-1}z)$$
on $\cC(\ol{\Ga})$ and $\pi_{\phi_n}$.

The pre-unitary structure on $ \cC(\ol{\Ga})$ is therefore given by the scalar products 
\begin{align*}
&B_n(f)=\|J_{n}(f)\|^2\\
&=\frac{1}{\mathcal{P}_{\gth,2}(\phi_n;\gs_n)}\sum_{g,h\in\Ga} f(g)\ol{f(h)}\phi_n(g^{-1}h)[\gth(d_\Ga(g,e))e^{-\gs_n\gd d_\Ga(g,e)}][\gth(d_\Ga(h,e))e^{-\gs_n\gd d_\Ga(h,e)}]\\
&=\iint f(g)\ol{f(h)}dm_n(g,h)
\end{align*}

where $m_n$ is a probability measure defined on the compact space $\ol{\Ga}\times \ol{\Ga}$ as
$$m_n(F)=\frac{1}{\mathcal{P}_{\gth,2}(\phi_n;\gs_n)}\sum_{\Ga\times \Ga}F(g,h)\phi_n(g^{-1}h)\gth(d_\Ga(g,e))e^{-\gs_n\gd d_\Ga(g,e)}\gth(d_\Ga(h,e))e^{-\gs_n\gd d_\Ga(h,e)}$$
for $F\in\cC(\ol{\Ga}\times\ol{\Ga})$.

\subsection{A measure theoretic perspective}\label{subsec:meas:pers}\hfill\break
In order to investigate the quadratic form $B_n$ introduced in Subsection \ref{subsec:twist:poincare} we consider $m$ a weak limit of sequence $(m_n)_n$ in $\tprob(\ol{\Ga}\times\ol{\Ga})$.

\begin{prop}\label{prop:fond:meas}
The measure $m$ is invariant by the flip of coordinates (i.e symmetric), $\Ga$-quasi-invariant, supported on $\dd \Ga\times \dd \Ga$ and satisfies
$$\frac{dg_*m}{dm}(\xi,\eta)=e^{-\gd({\boldsymbol\phi})  b_\xi(g^{-1})}e^{-\gd({\boldsymbol\phi})  b_\eta(g^{-1})}$$
for all $g\in\Ga$.
Moreover $m(\text{Diag}(\dd \Ga))=0$ whenever $\gd({\boldsymbol\phi})>\half\gd$.
\end{prop}
In order to prove Proposition \ref{prop:fond:meas} we need the following Lemma:

\begin{lem}
Let $(m_n)$ and $(m_n')_n$ be two sequences of probability measure on a compact metric space $Z$ such that $m_n\rightarrow m$ and $m_n'\rightarrow m'$ for the weak topology.
Assume $m_n$ is absolutely continuous with respect to $m_n'$ and the Radon-Nikodym derivatives $f_n=\frac{dm_n}{dm_n'}$ are continuous and converge uniformly to a function $f$.
Then $m$ is absolutely continuous with respect to $m'$ and $f=\frac{dm}{dm'}$.
\end{lem}
\begin{proof}
As $f\in \cC(Z)$ for any $F\in\cC(Z)$, there exists $N$ such that $\|f-f_n\|_\infty<\e\|F\|^{-1}_\infty$ and $|m(fF)-m_n(fF)|<\e$ for all $n\ge N$.
It follows
\begin{align*}
|m(fF)-m_n'(F)|
&=|m(fF)-m_n(f_nF)|\\
&\le |m_n(|f_n-f|.|F|)|+|m(fF)-m_n(fF)|\\
&\le \|F\|_\infty\|f_n-f\|_\infty+|m(fF)-m_n(fF)|\le 2\e
\end{align*}
In other words $m_n'(F)\rightarrow m(fF)$ and thus $m(fF)=m'(F)$ for all $F\in\cC(Z)$.
\end{proof}

\begin{proof}[Proof of Proposition \ref{prop:fond:meas}]
The symmetry follows from the construction.
Let $(\pi_\phi,\mathcal{H}_\phi,{\bf1}_\phi)$ be the positive cyclic representation associated to $\phi$.

Given a neighbourhood, $W$, of $\dd\Ga\times \dd\Ga$ one can find a finite set, $F$, of $\Ga$ such that $W^c\subset \ol{\Ga}\times F\cup F\times \ol{\Ga}$.
Indeed otherwise one can construct a sequence $(x_n,y_n)\notin W$ with  $x_n,y_n\rightarrow\dd \Ga$.

Let $F\in\cC(\ol{\Ga}\times \ol{\Ga})$ with $\|F|_{W}\|_\infty\le \e$, then
\begin{align*}
&|m_n(F)-m_n(F|_{W})|\le \|F\|_\infty m_n(\ol{\Ga}\times F\cup F\times \ol{\Ga})\\
&\le\frac{2\|F\|_\infty }{\mathcal{P}_{\gth,2}(\phi_n;\gs_n)}\sum_{F} \gth(d_\Ga(g,e))e^{-\gs_n\gd d_\Ga(g,e)}(\pi_{\phi_n}(g){\bf1}_{\phi_n},\sum_\Ga \gth(d_\Ga(h,e))e^{-\gs_n\gd d_\Ga(h,e)}\pi_{\phi_n}(h){\bf1}_{\phi_n})\\
&\le_\text{Cauchy-Schwarz} \frac{2\|F\|_\infty}{\sqrt{\mathcal{P}_{\gth,2}(\phi_n;\gs_n)}}\sum_{g\in F}  \gth(d_\Ga(g,e))e^{-\gs_n\gd d_\Ga(g,e)}\xrightarrow{n\rightarrow+\infty} 0
\end{align*} 
as $F$ is finite and $\mathcal{P}_{\gth,2}(\phi_n;\gs_n)\rightarrow+\infty$.
It follows that $m$ is supported in $\dd\Ga\times \dd \Ga$.

Using strong hyperbolicity the functions
$$D_\gs(\gamma;g,h)=\frac{\gth(d_\Ga(g,\gamma))}{\gth(d_\Ga(g,e))}e^{-\gs\gd b_g(\gamma,e)}
\frac{\gth(d_\Ga(h,\gamma))}{\gth(d_\Ga(h,e))}e^{-\gs\gd b_h(\gamma,e)}$$
on $\Ga\times\Ga$  extend continuously to $\ol{\Ga}\times\ol{\Ga}$
as 
\[   
D_\gs(\gamma;\xi,h)=
     \begin{cases}
       e^{-\gs\gd b_\xi(\gamma,e)}
\frac{\gth(d(h,\gamma))}{\gth(d(h,e))}e^{-\gs\gd b_h(\gamma,e)} &\quad \text{on $\dd\Ga\times \Ga$}\\
       e^{-\gs\gd b_\xi(\gamma,e)}e^{-\gs\gd b_\eta(\gamma,e)} &\quad\text{on $\dd\Ga\times \dd\Ga$.} 
     \end{cases}
\]

A direct computation shows that
$$\frac{d\gamma_*m_n}{dm_n}(g,h)
=D_{\gs_n}(\gamma;g,h)$$
on $\Ga\times\Ga$ for all $\gamma\in\Ga$.

The sequences of probability measures $(m_n)_n$ and $(\gamma_*m_n)_n$ converge respectively to $m$ and $\gamma_*m$. 
Moreover $D_{\gs_n}(\gamma;)$ converges uniformly to $D_{s({\boldsymbol\phi})}(\gamma;\bullet)$ for all $\gamma\in\Ga$. 
It follows that $m$ is $\Ga$-quasi-invariant and 
$$\frac{dg_*m}{dm}(\xi,\eta)
=e^{-\gs\gd b_\xi(\gamma,e)}e^{-\gs\gd b_\eta(\gamma,e)}$$
as $m$ is supported on $\dd\Ga\times \dd\Ga$.

Assume $s>\half$ and $m(\text{Diag}(\dd\Ga))\neq0$, as $m$ is quasi-invariant and $\dd\Ga$ minimal  $\frac{1}{m(\text{Diag}(\dd\Ga))}m|_{\dd\Ga}$ is a $2s\gd$-conformal probabilities supported on $\dd\Ga$ with $2s\gd>\gd$ which contradict the uniqueness of the conformal measure for hyperbolic groups  (see \cite[Thm. 7.7]{MR1214072}) .
\end{proof}

Note that $f\mapsto m(f\otimes f)^\half$ is a $\pi_s$-invariant (with $s\gd=s({\boldsymbol\phi})\gd=\gd({\boldsymbol\phi})$) semi-norm on $\cC(\dd \Ga)$ and thus satisfies the triangle inequality.
Together with the conformal properties of $m$ it allows us to recover a partial shadow Lemma:
\begin{lem}[Upper shadow lemma]\label{lem:half:shadow}
Given $f\in\cC(\dd\Ga)$ a function supported on a Shadow $\mathcal{O}(\gamma,r_0)$.
Then $$m(f\otimes f)^\half\prec  e^{2\gd({\boldsymbol\phi})r_0}e^{-\gd({\boldsymbol\phi})d_\Ga(\gamma,e)}\|f\|_\infty$$

\end{lem}
This Lemma is used in Theorem's \ref{thm:amenable:gene} proof.
\begin{proof}
For $\xi\in \mathcal{O}(\gamma,r_0)$ one has:
\begin{align*}
b_{\gamma^{-1}\xi}(\gamma^{-1},e)+d_\Ga(\gamma,e)&=
-b_{\xi}(\gamma,e)+d_\Ga(\gamma,e)\\
&=2(\xi,\gamma)\ge  2d_\Ga(\gamma,e)-2r_0
\end{align*} 
and thus $b(\gamma^{-1},e)\ge d_\Ga(\gamma,e)-2r_0$ on $\gamma^{-1}\mathcal{O}(\gamma,r_0)$.
It follows 
\begin{align*}
m(f{\bf1}_{\mathcal{O}(\gamma,r_0)}{\otimes}f{\bf1}_{\mathcal{O}(\gamma,r_0)})^\half
&=m((\pi_s(\gamma^{-1})f).{\bf1}_{\gamma^{-1}\mathcal{O}(\gamma,r_0)}{\otimes}(\pi_s(\gamma^{-1})f).{\bf1}_{\gamma^{-1}\mathcal{O}(\gamma,r_0)})^\half\\
&\le e^{2\gd({\boldsymbol\phi})r_0}e^{-\gd({\boldsymbol\phi})d_\Ga(\gamma,e)}m(g^*f{\otimes}g^*f)^\half\\
&\le e^{2\gd({\boldsymbol\phi})r_0}e^{-\gd({\boldsymbol\phi})d_\Ga(\gamma,e)}\|f\|_\infty
\end{align*}
\end{proof}

The $\Ga$-invariant coordinate projection maps $p_i:\dd\Ga\times \dd\Ga\rightarrow \dd\Ga$ ($i=1,2$)  induce a disintegration of the $\Ga$-quasi-invariant probability measure $m$ over the compact set $\dd\Ga$. As $m$ is symmetric we denote $(p_1)_*m\cong_\text{isom.}(p_2)_*m=\nu$.

As a consequence of the desintegration theorem (see \cite[Chapter \romannumeral 3]{MR488194} or \cite{MR1799683}) there exist a family on probability measures, $(\rho_\xi)_{\xi\in\dd \Ga}$  on $\dd\Ga$ and a multiplicative cocycle $c:\Gamma\times \dd \Ga\rightarrow\BR_+$ such that for any positive Borel function $\vp$ on $\dd \Ga$
$\xi\mapsto \rho_\xi(\vp)$ is also Borel,
$$\int \vp(\gamma^{-1}\eta)c(\gamma,\eta)d\rho_\xi(\eta)=\int \vp(\eta)d\rho_{\gamma^{-1} \xi}(\eta)$$
, i.e the probability $\gamma_*\rho_\xi$ and $\rho_{\gamma^{-1} \xi}$ are equivalent,
and
$$\iint_{\dd\Ga\times\dd\Ga} \Phi dm=\int_{\dd\Ga}\int_{\dd\Ga} \Phi(\xi,\eta)d\rho_\xi(\eta)d\nu(\xi)$$
for all positive measurable function $\Phi\in L^0_+(\dd \Ga\times \dd \Ga)$.

The family $(\rho_\xi)_\xi$ induces $\Ga$-equivariant conditional expectation, $E$, on $(\dd \Ga,\nu)$ given by 
$$E(\Phi)_\xi=\int_{\dd\Ga} \Phi(\xi,\eta)d\rho_\xi(\eta)$$ for all positive Borel function on $\dd\Ga\times\dd \Ga$
such that
$$\int f(\xi)f'(\eta)dm(\xi,\eta)=\int f(\xi)E_\xi(f')d\nu(\xi)$$ 
for all $f,f'\in L^0_+(\dd X)$.
In particular $E$ induces a contraction on every $L^p(\nu)$ for $1\le p\le\infty$ and satisfies the monotone convergence theorem.

\begin{defn}\label{def:gamma:s:cond}
Given $\mu$ a quasi-invariant probability measure on $\ol{\Ga}$, a conditional expectation, $E$, is called $(\Ga,s,\mu)$-conformal if 
$$E(\pi_s(g)f)=\pi_{s}^*(g)E(f)$$
for $f\in L^2(\ol{\Ga})$ and $\pi_{s}^*(g)$ the adjoint of $\pi_{s}(g)$ in $L^2(\ol{\Ga})$.
\end{defn}

\begin{prop}[Generalized Knapp-Stein relation]\label{prop:gene:knap:stein}
For all $1\le p\le\infty$, the conditional expectation $E$ induces a contraction on $L^p(\dd \Ga,\nu)$ that is $(\Ga,s,\mu)$-conformal.
Moreover $E$ is positive definite on $L^2(\dd \Ga,\nu)$.
\end{prop}

\begin{proof}
The contraction property and the positivity on $L^2(\dd \Ga,\nu)$ follows respectively from general facts about conditional expectations \cite{MR488194} and the origin of the measure  $m$.

A direct computation shows:
$$\pi_{s({\boldsymbol\phi})}^*(g)\psi(\xi)=e^{\gd({\boldsymbol\phi}) b_\xi(g)}\frac{dg^{-1}_*\nu}{d\nu}(\xi)\psi(g^{-1}\xi)$$

Given $\psi_1,\psi_2\in L^2(\dd \Ga)$ one has:
\begin{align*}
&\int \psi_1(\xi)\int e^{-\gd({\boldsymbol\phi}) b_\eta(g)}\psi_2(g^{-1}\eta)d\rho_\xi(\eta)d\nu(\xi)\\
&=\int \psi_1(g\xi)e^{\gd({\boldsymbol\phi}) b_\eta(g^{-1})}\psi_2(\eta)d(g^{-1})_*m(\xi,\eta)
=\int e^{-\gd({\boldsymbol\phi}) b_\xi(g^{-1})}\psi_1(g\xi)E(\psi_2)(\xi)d\nu(\xi)\\
&=\int \psi_1(\xi)\frac{dg^{-1}_*\nu}{d\nu}(\xi)e^{\gd({\boldsymbol\phi}) b_\xi(g)}E(\psi_2)(g^{-1}\xi)d\nu(\xi)
\end{align*}

In other words
$$E(e^{-\gd({\boldsymbol\phi}) b(g)}g^*\psi)(\xi)=e^{\gd({\boldsymbol\phi}) b_\xi(g)}\frac{dg^{-1}_*\nu}{d\nu}(\xi)E(\psi)(g^{-1}\xi)$$
for all $g\in\Ga$ and $\psi\in\cC(\dd \Ga)$.

\end{proof}

\begin{rem}
Note that 
$$\frac{dg_*\nu}{d\nu}(\xi)=e^{-\gd({\boldsymbol\phi})b_\xi(g)}E(e^{-\gd({\boldsymbol\phi}) b_\bullet(g)})(\xi)$$
$\nu$-almost surely and thus
$$E(\pi_s(g)\psi)=E(e^{-\gd({\boldsymbol\phi}) b_\bullet(g)})E(\psi)$$
for all $\psi\in L^1(\nu)$ and $g\in \Ga$.
\end{rem}

\subsection{Back to quadratic forms}\label{subsec:quad:final}\hfill\break
It follows from Subsection \ref{subsec:meas:pers} that $B_{\boldsymbol\phi}^\dd(\psi)\df(E(\psi),\psi)_{L^2(\nu)}$ is a positive $\pi_s$-invariant (with $s\gd=s({\boldsymbol\phi})\gd=\gd({\boldsymbol\phi})$) quadratic form on $L^2(\dd\Ga,\nu)$.

Observe that the operators $J_n:\cC(\ol{\Ga})\rightarrow\mathcal{H}_{\phi_n}$ (introduced in Subsection \ref{subsec:twist:poincare}) induce the maps $E_n=J_n^*J_n:\cC(\ol{\Ga})\rightarrow\cC(\ol{\Ga})'$ such that $(E_n(f_1),f_2)=B_n(f_1,f_2)$ for all $f_1,f_2\in \cC(\ol{\Ga})$.

\begin{lem}\label{lem:limit:weak}
The sequence of operators $(E_n)_n$ converges to $E$ for weak operator topology (see Appendix \ref{appendix:WOT}).
In particular 
$$B_n(f_1,f_2)\rightarrow B_{\boldsymbol\phi}^\dd(f_1|_{\dd \Ga},f_2|_{\dd \Ga})$$
for all $f_1,f_2\in\cC(\ol{\Ga})$.

\end{lem}
\begin{proof}
Given $f_1,f_2\in\cC(\ol{\Ga})$ one has:
\begin{align*}
\lim_n(E_n(f_1),f_2)&=\iint f_1(x)\ol{f_2}(y)dm_n(x,y)\\
&\rightarrow \iint f_1(\xi)\ol{f_2}(\eta)dm(\xi,\eta)=\int_{\dd \Ga} f_1(\xi)E(f_2)_\xi d\nu(\xi)
\end{align*}

\end{proof}

As $\pi_{\gth,\gs_n}(g)f$ converges uniformly to $\pi_{s}(g)f$ for all $f\in\cC(\ol{\Ga})$ it follows 
$$(E_n(\pi_{\gth,\gs_n}(g)f_1),f_2)\rightarrow (E(\pi_s(g)f_1),f_2)$$
for all $f_1,f_2\in\cC(\ol{\Ga})$ (see Lemma \ref{lem:app:wot:comp} Appendix \ref{appendix:WOT}).


\section{First properties}\label{sec:first:prop}
Given $\phi\in\mathbb{P}_+(\Ga)$ and $(\pi_\phi,\mathcal{H}_\phi,{\bf1}_\phi)$ its GNS triple, we denote $(\pi_\phi^\dd,\mathcal{H}_\phi^\dd,{\bf1})$ a $\phi$-boundary representation constructed in Section \ref{sec:const:bound}, $B_n$ with $n\ge0$ and $B_\phi^\dd$ the quadratic forms considered respectively Subsections \ref{subsec:twist:poincare}  and  \ref{subsec:quad:final} (in this case $\phi_n=\phi$ for all $n\ge0$ and $\gd({\boldsymbol\phi})=\gd(\phi)$).
We assume $\gth=1$ (see Lemma \ref{lem:patt:arg}) whenever no particular adjustment are needed.

\subsection{The relations between $\pi_\phi^\dd$ and $\pi_\phi$}\label{subsec:weak:cont:andco}\hfill\break

A representation is mixing if its matrix coefficients belongs to $c_0(\Ga)$.
The mixing property is invariant by isomorphism.

\begin{lem}
The representation $\pi_\phi^\dd$ factorises through the kernel of $\phi$.
In particular $\pi_\phi^\dd$ is not mixing whenever $\phi$ has an infinite kernel.
\end{lem}
\begin{proof}
For $\gs>s(\phi)$, $g\in \text{Ker}(\pi_\phi)$ and $f_1,f_2\in\cC(\ol{\Ga})$ one has:
\begin{align*}
&B_n(\pi_\gs(\gamma)f_1,f_2)=\sum_{g,h} e^{-\gs\gd b_g(\gamma)}f_1(\gamma^{-1}g)f_2(h)\phi(g^{-1}h)e^{-\gs\gd d_\Ga(g,e)}e^{-\gs\gd d_\Ga(h,e)}\\
&=\sum_{g,h} f_1(g)f_2(h)\phi(g^{-1}h[h^{-1}\gamma^{-1}h])e^{-\gs\gd d(g,e)}e^{-\gs\gd d(h,e)}\\
&=B_n(f_1,f_2)
\end{align*}
On the other hand, using Lemma \ref{lem:limit:weak}:
$$B_n(\pi_{\gs_n}(g)(f_1),f_2)=(E_n\pi_{\gs_n}(g)(f_1),f_2)\rightarrow (E(\pi_s(g)f_1),f_2)$$
It follows that $(E(\pi_s(g)f_1),f_2)=(E(f_1),f_2)$ and thus $\pi_\phi^\dd(g)=1$ for all $g\in \text{Ker}(\pi_\phi)$ as $\cC(\ol{\Ga})$ generates a dense subspace in $\mathcal{H}_\phi^\dd$.

\end{proof}

\begin{rem}
All the boundary representations considered in  \cite{MR2787597} \cite{Garncarek:2014aa} \cite{Boucher:2020ab} but also every spherical representations of semi-simple Lie groups with finite center are mixing (due to Howe-Moore theorem \cite{MR1781937}).
It follows that $\pi_\phi^\dd$  does not identify with one of them whenever $\phi$ has an infinite kernel, in other words $\pi_\phi^\dd$ is of exotic type.
\end{rem}

A unitary representation $\pi$ is weakly contained in $\pi'$ if every $\pi$-matrix coefficient is a limit for the compact-open topology of $\pi'$-matrix coefficients and  denoted $\pi\prec \pi'$ (see \cite{Bekka:2019aa} and \cite{MR2415834} appendix F for more about weak containment).

\begin{lem}\label{lem:seq:weak:cont}
Let ${\boldsymbol\phi}=(\phi_n)_n$ be a sequence of positive cyclic representations such that $\lim_n\gd(\phi_n)=\gd({\boldsymbol\phi})$ is well defined and $\pi_{\boldsymbol\phi}^\dd$ an associated boundary representation.
Then $\pi_n\rightarrow \pi_{\boldsymbol\phi}^\dd$ for the Fell topology.
In particular given $\phi\in\mathbb{P}_+(\Ga)$, $\pi_\phi^\dd$ is weakly contained in $\pi_\phi$ and $\|\pi_\phi^\dd(f)\|\le\|\pi_\phi(f)\|$ for all $f\in\BC[\Ga]$.
\end{lem}
\begin{proof}
Let $[\pi_n,v_n]$ be the positive cyclic representation associated to $\phi_n$.
As explained in Subsection \ref{subsec:twist:poincare} $\pi_{\gth,\gs_n}\simeq_\text{isom.} \pi_n$ for all $n\in\BN$ and as the $\pi^\dd$ matrix coefficients are limits of $(\pi_{\gth,\gs_n})_n$ coefficients due to Lemma \ref{lem:limit:weak} (and \cite[Lem. F.2.2]{MR2415834}) it follows that
$\pi_{\gs_n,\gth}\simeq\pi_n\rightarrow\pi^\dd$.
\end{proof}

This can be improved assuming the Poincar\'e series $\mathcal{P}(\phi)$ and $\mathcal{P}_{2}(\phi)$ (see Subsection \ref{subsec:crit:poin}) satisfy a certain inequality:\\
As the space $\text{B}_{\le 1}(\cC(\ol{\Ga}),\mathcal{H})$ of bounded operators with norms at most $1$ is compact for the weak operator topology (see Appendix \ref{appendix:WOT}) one can assume that the sequence $(J_{\gth,\gs_n})_n$ in $\text{B}_{\le 1}(\cC(\ol{\Ga}),\mathcal{H})$ introduced in Subsection \ref{subsec:twist:poincare} converges to an operator $J_\phi$ for this topology.
In order to prove that $J$ is non-zero it is enough to have 
$$(J_{\gth,\gs_n}{\bf1}_{\ol{\Ga}},{\bf1}_\phi)=\frac{\mathcal{P}_{\gth}(\phi;\gs_n)}{\sqrt{\mathcal{P}_{\gth,2}(\phi;\gs_n)}}\ge c>0$$
for $n$ large enough. We call this condition \textit{Poincar\'e sphericity}.
This leads to the following criterion:

\begin{prop}\label{prop:poin:sph}
If $\phi$ is Poincar\'e spherical 
then there exists a non-trivial intertwiner between $\pi_\phi^\dd$ and $\pi_\phi$.
\end{prop}

\begin{proof}

We are going to prove that the limit $J_\phi$ is an intertwiner.
Given $f\in\cC(\ol{\Ga})$
\begin{align*}
&(J_{\gth,\gs_n}\pi_{\gth,\gs_n}(g)f,v)=(J_{\gth,\gs_n}f,\pi_\phi(g^{-1})v)\\
&\rightarrow (J_\phi f,\pi_\phi(g^{-1})v)=(\pi_\phi(g)J_\phi f,v)
\end{align*}
On the other hand, as $\pi_{\gth,\gs_n}(g)f$ converges uniformly to $\pi_s(g)f$, Lemma \ref{lem:app:wot:comp} implies:
\begin{align*}
&(J_{\gth,\gs_n}\pi_{\gth,\gs_n}(g)f,v)\rightarrow (J_\phi\pi_s(g)f,v)
\end{align*}
Since $\cC(\ol{\Ga})$ is $\pi_{\gth,\gs_n}$-invariant dense subspace of $\mathcal{H}_\phi^\dd$ and $(J_\phi{\bf1}_{\ol{\Ga}},{\bf1}_\phi)>0$ it follows that $J_\phi$ is a non-zero intertwiner.

\end{proof}

The Poincar\'e sphericity condition can be deduced when $\phi$ is roughly radial:
\begin{prop}\label{prop:hs:rough}
Let $f$ be a pointwise positive function and assume that $f$ is $k_s$-roughly radial function with $k_s(t)= e^{-(1-s)\gd t}$ for $\half\le s\le 1$.
Then $\mathcal{P}_2(f;\gs)\asymp\mathcal{P}^2(f;\gs)$ for all $\gs>s(f)$.
\end{prop}
\begin{proof}
On a Gromov hyperbolic groups the functional 
$\gs\mapsto \sum_{g\in\Ga}e^{-\gs \gd d_\Ga(g,e)}$ for $\gs >1$ has a expansion, when $\gs$ approaches $1^+$, comparable to $\frac{1}{(\gs-1)}$ due to pure sphericality \cite[Thm. 7.2]{MR1214072}.
As $$\sum_g\phi(g)e^{-(s+\e)\gd d_\Ga(g,e)}\asymp \sum_gk_s(d_\Ga(g,e))e^{-(s+\e)\gd d_\Ga(g,e)}\asymp \sum_ge^{-(1+\e)\gd d_\Ga(g,e)}\asymp \e^{-1}$$
for all $\e>0$ it follows that $\phi$ is of divergent type.
Using Lemma \ref{lem:tech:v1}, Remark \ref{rem:tech:half} and their notations:
\begin{align*}
\mathcal{P}_2(f;\gs)&=\sum_{g,h}e^{-\gs \gd d_\Ga(g,e)}e^{-\gs \gd d_\Ga(h,e)}f(g^{-1}h)\\
&=_{h'=g^{-1}h}\sum_{h'}f(h')e^{-\gs \gd d_\Ga(h',e)}\sum_ge^{-\gs \gd d_\Ga(g,e)}\pi_\gs(g){\bf1}_{h'}\\
&=\sum_{h}f(h)e^{-\gs \gd d_\Ga(h,e)}(\sum_{m:Rm\le d_\Ga(h,e)}+\sum_{m:d_\Ga(h,e)<Rm})\mathcal{S}_m(\gs,\gs;h)\\
&\asymp \sum_{h\in\Ga} f(h)e^{-\gs\gd d_\Ga(h,e)}(d_\Ga(h,e)+1) 
\asymp \sum_{h\in\Ga} e^{-(1+(\gs-s))\gd d_\Ga(g,e)}(d_\Ga(h,e)+1) \\
&\asymp \frac{1}{(\gs-s)^2}
\end{align*}
The last estimate follows from:
\begin{align*}
 \sum_{h\in\Ga}& e^{-(1+(\gs-s))\gd d_\Ga(g,e)}(d_\Ga(h,e) +1)
 =\sum_{m=0}^\infty\sum_{h\in C_m} e^{-(1+(\gs-s))\gd d_\Ga(g,e)}(d_\Ga(h,e)+1) \\
&\asymp_{(1)} \lim_N\frac{-1}{1-e^{-(\gs-s)\gd R}}\sum_{m\le N} (e^{-(\gs-s)\gd R(m+1)}-e^{-(\gs-s)\gd Rm})m \\
&\asymp_{(2)}\lim_N\frac{-1}{1-e^{-(\gs-s)\gd R}}e^{-(\gs-s)\gd R(m+1)}(m+1)+\frac{1}{1-e^{-(\gs-s)\gd R}}\sum_{1\le m\le N} e^{-(\gs-s)\gd Rm} \\
&\asymp[\frac{1}{1-e^{-(\gs-s)\gd R} }]^2\asymp_{\gs\sim s^+}\frac{1}{(\gs-s)^2}
\end{align*}
where $(1)$ is due to the pure sphericality of $\Ga$ and $(2)$ a summation by part.
One has shown that $\mathcal{P}_2(f;\gs)\asymp\mathcal{P}^2(f;\gs)$ for $\gs\sim s(f)^+$.
\end{proof}

\subsubsection{About the boundary representation of a boundary representation}\hfill\break
Given a $\phi$-boundary representation, $\pi_\phi^\dd$, we denote $\Xi_\phi^\dd=(\pi_\phi^\dd{\bf1},{\bf1})$ its Harish-Chandra function namely the generating function of the positive cyclic representation $(\pi_\phi^\dd,\mathcal{H}_\phi^\dd,{\bf1})$.

\begin{prop}
The function $\Xi_\phi^\dd\in\mathbb{P}_+(\Ga)$ is Poincar\'e spherical of divergent type and exponent $\max\{\gd(\phi),\half\gd\}$.
In particular there exists a non-zero intertwiner between $\pi_\phi^\dd$ to $\pi_{\Xi_\phi^\dd}^\dd$.
\end{prop}
\begin{proof}
Recall that the conformal exponent of $\pi_\phi^\dd$ is equal to $\max\{s(\phi),\half\}$ due to Proposition \ref{prop:two:div}.

Using Proposition \ref{prop:gene:knap:stein} notations and Corollary \ref{cor:tech:sum:last} estimates, as $\nu$ is supported on $\dd \Ga$,  one has:

\begin{align*}
\mathcal{P}(\Xi_\phi^\dd;\gs)=\sum_{g\in\Ga}e^{-\gs\gd d_\Ga(g,e)}\Xi_\phi^\dd(g)&=(E(\mathcal{S}_\infty(\gs,s;\bullet)),{\bf1})_{L^2(\nu)}\\
&\asymp_s\frac{1}{(\gs-s)}
\end{align*}
and 
\begin{align*}
\mathcal{P}_2(\Xi_\phi^\dd;\gs)=(E(\mathcal{S}_\infty(\gs,s;\bullet)),\mathcal{S}_\infty(\gs,s;\bullet))_{L^2(\nu)}\asymp_s\frac{1}{(\gs-s)^2}
\end{align*}
for $s>\half$.
Similarly for $s=\half$ :
$$\mathcal{P}(\Xi_\phi^\dd;\gs)\asymp\frac{1}{(\gs-\half)^2}\quad \text{and}\quad \mathcal{P}_2(\Xi_\phi^\dd;\gs)\asymp\frac{1}{(\gs-\half)^4}$$

It follows that $\max\{\gd(\phi),\half\gd\}=\gd(\Xi_\phi^\dd)$, $\Xi_\phi^\dd$ is of divergent type and Poincar\'e spherical.
The rest follows from Proposition \ref{prop:poin:sph}.

\end{proof}

\begin{rem}
It is reasonable to conjecture that $\pi_\phi^\dd$ is irreducible.
In this case Poincar\'e sphericity implies that $\pi_\phi^\dd$ is realised as a subrepresentation of $\pi_\phi$.
In particular this would imply that $\pi_\phi^\dd$ and $\pi_{\Xi_\phi^\dd}^\dd$ are isomorphic.
\end{rem}

\subsection{Natural complementary series}\label{subsec:nat:comp}\hfill\break

In certain situations the invariant scalar product, $B_\phi^\dd$, can be explicitly described.\\
Given $t>\gd$ we denote $\mu_t$ the probability measure:
$$\mu_t=\frac{1}{\mathcal{P}({\bf1}_\Ga,\gd^{-1}t)}\sum_{g\in\Ga}e^{-td_\Ga(g,e)}\text{Dir}_{g}$$
on $\ol{\Ga}$ where $\text{Dir}_{g}$ stands for the Dirac measure at $g\in \Ga$.
As a consequence of strong hyperbolicity one has $\mu_t\xrightarrow{t\rightarrow \gd^+}\nu_\text{PS}$ where $\nu_\text{PS}$ stands for the Patterson-Sullivan measure on $\dd \Ga$.
Indeed strong hyperbolicity implies that any limit of $(\mu_t)_{t>\gd}$ must be a $(\Ga,d_\Ga)$-conformal probability but this probability is unique (see \cite[Thm. 7.7]{MR1214072}).

\begin{prop}\label{prop:nat:comp}
Let $f$ be a pointwise positive function that is $k_s$-roughly radial function with
$k_s(t)= e^{-(1-s)\gd t}$ ($\half<s\le 1$), $m$ and $\nu$ be the measures considered in Subsection \ref{subsec:meas:pers}.
Then $f$ is of divergent type, the measure $\nu$ and $\nu_\text{PS}$ are equivalent with $m=c(s).m_s$ where
$$m_s(\Phi)=\iint\Phi(\xi,\eta)e^{2(1-s)\gd (\xi,\eta)}d\nu_\text{PS}(\xi) d\nu_\text{PS}(\eta)$$
 and $c(s)=(m_s({\bf1}_{\dd \Ga\times\dd \Ga}))^{-1}$ for all $\Phi\in\cC(\dd \Ga\times\dd \Ga)$.
\end{prop}

In order to prove Proposition \ref{prop:nat:comp} we need the following results proved in \cite{Boucher:2020ab} (see Lemma 2.4 and Corollary 4.1):

\begin{lem}[sequential shadow lemma]\label{lem:seq:shadow:bou}
There exist $\gd<\gd'$ and $r_0$ such that:
$$\mu_t(\{x\in\ol{\Ga}:(x,h)\ge d_\Ga(h,e)-r\})\asymp e^{t r}e^{-t d_\Ga(h,e)} $$
for all $\gd<t<\gd'$, $h\in \Ga$ and $r>r_0$.
\end{lem}

\begin{cor}
The sequence of probabilities
$$dm_{s;t}(g,h)=e^{2(1-s)\gd(g,h)}d\mu_t(g)d\mu_t(h)$$
converges weakly to 
$$dm_s(\xi,\eta)=e^{2(1-s)\gd(\xi,\eta)}d\nu_\text{PS}(\xi)d\nu_\text{PS}(\eta)$$
on $\ol{\Ga}\times\ol{\Ga} $.
\end{cor}

\begin{rem}
Even if those results was originally stated in a CAT(-1) context their proofs adapt to proper roughly geodesic strongly hyperbolic spaces.
\end{rem}

\begin{proof}[Proof of Proposition \ref{prop:nat:comp}]
It follows from Proposition \ref{prop:hs:rough} that 
$f$ is of divergent type and $\mathcal{P}_2(f;\gs)\asymp\mathcal{P}^2(f;\gs)$ for all $\gs\sim s(f)^+$.

Given $\Phi\in \cC_+(\ol{\Ga}\times \ol{\Ga})$ and  $\gs>s(f)$ one has:
\begin{align*}
&\frac{1}{\mathcal{P}_2(f;\gs)}\iint_{\Ga\times\Ga} \Phi(g,h)\phi(g^{-1}h)e^{-\gs \gd d(g,e)}e^{-\gs \gd d(h,e)}dhdg\\
&\prec \frac{1}{\mathcal{P}_2(f;\gs)}\iint_{\Ga\times\Ga} \Phi(g,h)e^{2(1-s)\gd (g,h)_e}e^{-(1+\gs-s)\gd d(g,e)}e^{-(1+\gs -s)\gd d(h,e)}dhdg\\
&\asymp\iint_{\Ga\times\Ga}\Phi(g,h)dm_{s;t}(g,h)
\end{align*}

Using Lemma \ref{lem:seq:shadow:bou} it follows
$m(\Phi)\le m_s(\Phi)$ and thus the non-zero functional
$$\vp:L^1(m_s)\rightarrow \BC;\quad \Phi\mapsto m(\Phi)$$
is a well defined, bounded and order preserving operator.

Therefore one can find a measurable positive and essentially bounded function $F$ on $\dd \Ga\times \dd \Ga$ such that
$$m(\Phi)=\vp(\Phi)=\iint_{\dd \Ga\times \dd \Ga}\Phi(\xi,\eta)F(\xi,\eta)e^{2(1-s)\gd (\xi,\eta)_e}d\nu_\text{PS}(\xi) d\nu_\text{PS}(\eta)$$

As $\vp(\pi_s\otimes\pi_s(g)\Phi)=\vp(\Phi)$ for all $\Phi\in \cC(\dd \Ga\times\dd \Ga)$ it follows that $g^*F=F$ and thus $F$ must be constant as the $\Ga$-action on $\dd \Ga$ is $2$-ergodic due to \cite[Thm. 4.1]{Coulon:2018aa} (see also \cite{Bader:2017aa} or \cite[Appendix A]{Garncarek:2014aa}).
\end{proof}

In this case the $(\Ga,s,\nu_\text{PS})$-conditional expectation is the renormalised Knapp-Stein operator
$$\mathcal{I}_s'(\psi)_\eta=(\go_1(\eta))^{-1}\int f(\xi)e^{2(1-s)(\xi,\eta)}d\nu_\text{PS}(\xi)$$
on $L^2(\go_1 d\nu_\text{PS})$
with $$\go(\eta)=\int e^{2(1-s)(\xi,\eta)}d\nu_\text{PS}(\xi)$$
and $\go_1=\frac{\go}{\nu_\text{PS}(\go)}$.

It follows from \cite[Lem 3.6]{Boucher:2020ab} that the Harish-Chandra functions: $$\Xi_s(g)\df (\mathcal{I}_s'(\pi_s(g){\bf1}),{\bf1})_{L^2(\go_1 d\nu_\text{PS})}$$ on $\Ga$ for $\half< s\le 1$ satisfy $\Xi_s(g)\asymp e^{-(1-s)\gd d_\Ga(g,e)}$.
We deduce, using Proposition \ref{prop:nat:comp}:
\begin{cor}
The Knapp-Stein of order $\half<s\le 1$ is positive definite iff there exists a positive definite function over $\Ga$ such that $\phi\asymp e^{-(1-s)\gd d_\Ga(\bullet,e)}$.
\end{cor}

\begin{cor}
The Knapp-Stein operators are positive definite for all $\half<s\le 1$ iff $\Ga$ has a proper affine action with unitary cocycle $b$ such that 
$\|b(g)\|^2\simeq_X c.d_\Ga(g,e)$
for a fixed $c>0$.
\end{cor}
\begin{proof}
Indeed if $\Ga$ admits such a affine action, then $\phi_t(g)\df e^{-t\|b(g)\|^2}\asymp  e^{-td_\Ga(g,e)}$ is roughly $k_t$-spherical with $k_t(x)=e^{-tx}$ with $\phi_t\in\mathbb{P}_+(\Ga)$ due to Schoenberg's theorem \cite[Prop. 2.11.1]{MR2415834} and Proposition \ref{prop:nat:comp} can be applied to prove the positivity of the Knapp-Stein operator.
The converse follows from \cite[Thm. 2]{Boucher:2020aa}.
\end{proof}

The following semi-group relation holds between the natural complementary series:
\begin{cor}
Given $\half<s,s',t\le 1$ with $s+s'=1+t$ and $\Xi_s,\Xi_{s'}\in \mathbb{P}_+(\Ga)$, one has $\Xi_t\in \mathbb{P}_+(\Ga)$ and $\pi_t\subset \pi_s\otimes\pi_{s'}$.
\end{cor}
\begin{proof}
The first part follows from the estimate
$\Xi_s.\Xi_{s'}(g)\asymp e^{-(1-t)\gd d_\Ga(g,e)}$ (see \cite[Lem 3.6]{Boucher:2020ab}) for all $g\in\Ga$ and Proposition \ref{prop:nat:comp}.

Proposition \ref{prop:poin:sph} together with Proposition \ref{prop:hs:rough} imply the existence of a non-zero intertwiner between $\pi_t$ and $\pi_{\Xi_s.\Xi_{s'}}\subset \pi_s\otimes\pi_{s'}$. On the other hand the irreducibility of $\pi_t$ \cite[Thm. 1]{Boucher:2020ab} (see also \cite{Boyer:2022aa}) implies the injectivity.
\end{proof}

\subsection{About the critical case $\gd=\gd(\phi)$}\label{subsec:crit:Rh:andco}\hfill\break

\begin{thm}\label{thm:amenable:gene}
Let ${\boldsymbol\phi}=(\phi_n)_n\in\mathbb{P}_+(\Ga)$, $([\pi_{\phi_n},{\bf1}_{\phi_n}])_n$ the sequence of positive cyclic representations associated and assume $\lim_n\gd(\phi_n)=\gd({\boldsymbol\phi})=\gd$, then $\pi^\dd\simeq1$ and $\pi_n\rightarrow 1$.
In particular given $\phi\in\mathbb{P}_+(\Ga)$ with $\gd(\phi)=\gd$, $1\prec\pi_\phi$.
\end{thm}

The proof of Theorem \ref{thm:amenable:gene} uses ideas coming from Proposition 5.21 and Corollary 5.22 of \cite{Coulon:2018aa}.

\begin{proof}
Let $(\pi^\dd, \mathcal{H}^\dd,\|.\|_s)$ be a boundary representation associated to $(\phi_n)_n$ and $([\pi_n,v_n])_n$ the cyclic representations associated.
As a consequence of Vitali's Lemma \ref{lem:vita} for all $n\in\BN$, there exists a continuous partition of the unit of $\dd \Ga$, $(\chi_\gamma)_{\gamma\in C_n^*}$, with $C_n^*\subset C_n$  such that
$$\mathcal{O}(\gamma,R)\subset \{\chi_\gamma\neq0\}\subset \mathcal{O}(\gamma,4R)$$
and $\chi_\gamma|_{\mathcal{O}(\gamma,R)}=1$ (see Subsection \ref{subsec:prelim:orbital} for notations).
Given $f\in\cC(\dd \Ga)$ Lemma \ref{lem:half:shadow} implies:
\begin{align*}
\|f\|_s\le \sum_{\gamma\in C_n^*}\|f\chi_\gamma\|_s=\sum_{\gamma\in C_n^*}m((f\chi_\gamma){\otimes}(f\chi_\gamma))^\half
\prec \sum_{\gamma\in C_n^*}e^{-\gd Rn}f_\gamma
\end{align*}
where $f_\gamma=\sup_{\mathcal{O}(\gamma,4R)}|f|$.

Using uniform continuity on the compact space $\dd \Ga$, for all $\e>0$, there exists $n_0$, for all $n\ge n_0$ and $\gamma\in C_n^*$, $|f_\gamma-f|\le \e$ on $\mathcal{O}(\gamma,4R)$.

On the other hand
\begin{align*}
\int_{\dd\Ga}|f|d\nu_\text{PS}\ge\sum_{\gamma\in C_n^*}\int_{\mathcal{O}(\gamma,R)}|f|d\nu_\text{PS}
\succ \sum_{\gamma\in C_n^*}e^{-\gd Rn}(f_\gamma-\e)
\succ \sum_{\gamma\in C_n^*}e^{-\gd Rn}f_\gamma-\e'
\end{align*}

and thus $\|f\|_{s}\le \|f|L^1(\nu_\text{PS})\|$.

This induces a bounded map:
$$T:L^1(\dd \Ga,\nu_\text{PS})\rightarrow \mathcal{H}^\dd$$  
and as Hilbert spaces are reflexive and thus have Radon-Nikodym property (see \cite[Chap. \romannumeral 3]{MR453964} ) it implies the existence of $K_T\in L^\infty(\dd \Ga,\nu_\text{PS};\mathcal{H}^\dd)$ such that
$$T(f)=\int_{\dd \Ga} f(\xi)K_T(\xi)d\nu_\text{PS}(\xi)\in \mathcal{H}^\dd$$
in the sense of Bochner integral.

In particular given any bounded map $V:\mathcal{H}^\dd \rightarrow B$ with range into a Banach space $B$, one has
$$V\circ T(f)=\int_{\dd \Ga} f(\xi)V(K_T(\xi))d\nu_\text{PS}(\xi)\in B$$
in the sense of Bochner integral (see \cite{MR453964} Chapter \romannumeral 2 \space Theorem 6).
It follows:
\begin{align*}
&(T(f_1),T(f_2))_{s}=\int_{\dd \Ga} f_1(\xi)(K_T(\xi),T(f_2))_{s}d\nu_\text{PS}(\xi)\\
&=\int_{\dd \Ga} f_1(\xi)\int_{\dd \Ga} f_2(\eta)(K_T(\xi),K_T(\eta))_{s}d\nu_\text{PS}(\eta)d\nu_\text{PS}(\xi)\\
&=\iint_{\dd\Ga\times \dd \Ga}f_1(\xi) f_2(\eta)\Phi_T(\xi,\eta)d\nu_\text{PS}(\eta)d\nu_\text{PS}(\xi)\\
\end{align*}
with $\Phi_T\in L^\infty(\dd \Ga\times \dd \Ga)$ and all $f_1,f_2\in L^1(\dd \Ga,\nu_\text{PS})$.

Note that the monotone class Lemma (see \cite[4.13]{MR2378491}) and the finiteness of $m_T=\Phi_T\nu_\text{PS}\otimes\nu_\text{PS}$ implies that $m_T$ is determined by its values on $L^1(\nu_\text{PS})\times L^1(\nu_\text{PS})$.

It follows from the $\pi_1$-invariant of the quadratic form $\|.\|_s^2$ that $\Phi_T$ is $\Ga$-invariant and thus constant, using the $2$-ergodicity of the measure $\nu_\text{PS}\otimes\nu_\text{PS}$ \cite[Thm. 4.1]{Coulon:2018aa}.
Eventually
$$\|f\|_s^2=|\int f(\xi) d\nu_\text{PS}(\xi)|^2$$ 
and using Lemma \ref{lem:seq:weak:cont} it follows $\pi_n\rightarrow 1\simeq\pi^\dd$.
\end{proof}

Theorem \ref{thm:amenable:gene} should be compared with \cite{Coulon:2018aa} Theorem 1.1:\\
Given a subset $S\subset \Ga$ the critical exponent of $S$ is defined as
$$\limsup_r\frac{1}{r}\log|S\cap B_\Ga(e,r)|$$
In particular $\gd(S)= \gd({\bf1}_S)$.

\begin{cor}
Let $\Lambda\subset \Ga$ be a subgroup.
If $\gd(\Lambda)= \gd({\bf1}_\Lambda)=\gd$ then $\Lambda$ is co-amenable in $\Gamma$.
\end{cor}

\subsection{About the regular representation}\hfill\break
The regular representation, $(\gl,\ell^2(\Ga))$, can be viewed as a cyclic representation with respect to the vector $\text{Dir}_e\in \ell^2(\Ga)$.
In this case $\phi_\gl={\bf1}_e$ and thus $\gd(\phi_\gl)=0$.
Observe that
$$\mathcal{P}_2(\phi_\gl,\gs)=\mathcal{P}(\phi_\gl,2\gs)\asymp\frac{1}{2\gs-1}$$
and for all $f\in\cC(\ol{\Ga})$:
\begin{align*}
B_{\gl,\gs}(f)=\iint |f(g)|^2 d\mu_{2\gs}(g)\xrightarrow{\gs\rightarrow\frac{\gd}{2}^+} \|f|L^2(\nu_\text{PS})\|^2
\end{align*}
It follows that the $\phi_\gl$-boundary representation is uniquely defined and corresponds to the Koopman representation on $L^2(\dd \Ga,\nu_\text{PS})$.

As a consequence of Lemma \ref{lem:seq:weak:cont}:
\begin{cor}\label{cor:tot:transf}
Let $(\pi_\half,L^2(\dd \Ga,\nu_\text{PS}))$ be the unitary Koopman representation over the boundary $(\dd \Ga,\nu_\text{PS})$, then $\pi_\half\prec\gl$. 
In particular $\|\pi_\half(f)\|\le\|\gl(f)\|$ for all $f\in \BC[\Ga]$.
\end{cor}
This fact is known by the mean of amenable actions (see \cite{MR1799683} for definitions):
\begin{thm*}[Kuhn \cite{MR1209424}]
Let $\Lambda$ be a discrete group, $\Lambda\curvearrowright (Z,\mu)$ an amenable action and $(\pi_\mu,L^2(Z,\mu))$ the associated Koopman representation.
Then $\pi_\mu\prec\gl$.
In particular $\|\pi_\mu(f)\|\le\|\gl(f)\|$ for all $f\in \BC[\Ga]$.
\end{thm*}
On the other hand it follows from Adams \cite{MR1293309} that the action of a Gromov hyperbolic group $\Ga$ on its boundary $(\dd \Ga,\mu)$ endowed with any quasi-invariant measure, $\mu$, is amenable.
Corollary \ref{cor:tot:transf} provides another perspective on this fact.

\section{Spectral bounds}\label{sec:spec:bound:all}

Let $\Lambda$ be a finitely generated group endowed with a $\Lambda$-invariant distance, $d_\Lambda$, quasi-isometric to a word distance.\\

We start with some elementary properties about critical exponents:
\begin{lem}\label{lem:elem:control}
Let $f$ be a pointwise positive function on $\Lambda$.
\begin{enumerate}
\item $\gd(f)\le (1-\frac{1}{p})\gd$ whenever $f\in\ell^p(\Lambda)$ for $p\in[1,+\infty]$.
\item For all $1\le p<\infty$, $\gd(f^p)\le \gd(f)$ whenever $f$ is bounded.
\item Given two positive finitely supported functions $h_1,h_2\in\BR_+[\Lambda]$ one has $\gd(h_1*f* h_2)=\gd(f)$.
\end{enumerate}
\end{lem}
\begin{proof}
Let $q$ such that $\frac{1}{p}+\frac{1}{q}=1$ and $\gs$ such that $q\gs>1$, i.e $\gs>1-\frac{1}{p}$.
Then
$$\sum_\Ga f(g)e^{-\gs\gd d_\Lambda(g,e)}\le [\sum_\Lambda |f(g)|^p]^\frac{1}{p}[\sum_\Lambda e^{-\gs q\gd d_\Lambda(g,e)}]^\frac{1}{q}$$
is finite.

For second part observe that replacing $f$ with $\frac{f}{\|f\|_\infty}$ we can assume $f\le 1$ and the result is a consequence of $x^p\le x$ on the interval $[0,1]$ for $1\le p<\infty$.

The third part follows from:
\begin{align*}
\sum_{g\in\Lambda}&h_1*f*h_2(g)e^{-\gs\gd d_\Lambda(g,e)}=\sum_{\Ga^{\times 3}} h_1(a)f(a^{-1}gb^{-1})h_2(b)e^{-\gs\gd d_\Lambda(g,e)}\\
&=\sum_{\Lambda^{\times 3}} h_1(a)f(g)h_2(b)e^{-\gs\gd d_\Lambda(agb,e)}\\
&\le(\sum_{a\in\Lambda} h_1(a)e^{\gs\gd d_\Lambda(a,e)})
(\sum_{b\in\Lambda} e^{\gs\gd d_\Lambda(b,e)}h_2(b))
(\sum_{g\in\Lambda} f(g)e^{-\gs\gd d_\Lambda(g,e)})\\
&\prec_{h_1,h_2}\mathcal{P}(f;\gs)
\end{align*}

\end{proof}

\subsection{Rapid decay and critical exponent}\hfill\break

The reader can refer to \cite{MR3666050} for more details about the rapid decay property and extra examples.
We denote $\gl:\Lambda\rightarrow \text{B}(\ell^2(\Lambda))$ the regular representation of $\Lambda$ on $\ell^2(\Lambda)$.\\

\begin{defn}
A discrete metric group $(\Lambda,d_\Lambda)$ has the rapid decay (RD) property, if there exist constants $C\in\BR_+$ and $m\in \BN$ such that for all $f\in \BC[\Lambda]$  
supported in a ball of radius $R$: $\|\gl(f)\|_\text{op}\le CR^m\|f\|_2$.
\end{defn}

\begin{thm}[Jolissaint \cite{MR943303}]
Gromov hyperbolic groups have the rapid decay property.
\end{thm}

We are going to use a weaker form of RD, namely the radial rapid decay property \cite{MR1488249}:\\
A function $f\in \BC[\Lambda]$ is called radial on $(\Lambda,d_\Lambda)$ if $f(g)=k(d_\Lambda(g,e))$ for some function $k$.

\begin{defn}\label{def:RRD}
A discrete metric group $(\Lambda,d_\Lambda)$ has the radial rapid decay (RRD) property, if there exist constants $C\in\BR_+$ and $m\in\BN$ such that for all radial function $f\in \BC[\Lambda]$ 
supported on a ball of radius $R$: $\|\gl(f)\|_\text{op}\le CR^m\|f\|_2$.
\end{defn}

\begin{rem}
It was recently proved in \cite{MR4613611} that RRD property is strictly weaker than RD.
\end{rem}

The RRD property can be used to estimate the critical exponent of certain representations:
\begin{prop}\label{prop:rrd:half}
Let $(\Lambda,d_\Lambda)$ with the radial rapid decay property and $\phi\in\mathbb{P}_+(\Ga)$.
If $\phi\prec \gl$ then $\gd(\phi)\le \frac{\gd}{2}$.
\end{prop}
\begin{proof}
As $(\Lambda,d_\Lambda)$ has RRD one can find $C>0$ and $k\in\BN$ as in Definition \ref{def:RRD}.
Note that for all $\kappa>0$, there exists $r_0$, for all $r\ge r_0$: $|B(e,r)|\le e^{(\gd+\kappa)r}$ and thus
\begin{align*}
\sum_{g\in B(e,r)}\phi(g)e^{-\gs \gd d_\Lambda(g,e)}&=\sum_{g\in B(e,r)}(\pi_\phi(g){\bf1}_\phi,{\bf1}_\phi)e^{-\gs\gd d_\Lambda(g,e)}\\
&\le\|\pi_\phi(\sum_{g\in B(e,r)}e^{-\gs \gd d_\Lambda(g,e)}\text{Dir}_g)\|_\text{op}\\
&\le_{\phi\prec\gl}\|\gl(\sum_{g\in B(e,r)}e^{-\gs\gd d_\Lambda(g,e)}\text{Dir}_g)\|_\text{op}\\
&\le_\text{RRD} Cr^m\|\sum_{g\in B(e,r)}e^{-\gs\gd d_\Lambda(g,e)}\text{Dir}_g\|_2
\prec r^me^{\half\kappa r}e^{(\half-\gs)\gd r}+1
\end{align*}
for all $r>r_0$.
It follows for $\gs>\half\gd$, $\kappa>0$ small enough and $r_0(\kappa)$ that
$$\sum_{g\in\Ga}\phi(g)e^{-\gs \gd d_\Lambda(g,e)}=\lim_r\sum_{g\in B(e,r)}\phi(g)e^{-\gs \gd d_\Lambda(g,e)}\prec \lim_rr^ke^{\half\kappa r}e^{(\half-\gs)\gd r}+1$$
is finite.
\end{proof}

\begin{rem}
As a consequence, in the hyperbolic case, the $\phi$-boundary representation $\pi_\phi^\dd$ is non-tempered (i.e $\pi_\phi^\dd\nprec \gl$) whenever $\gd(\phi)>\half\gd$.\\
The author don't know whether $\gd(\phi)\le\frac{\gd}{2}$ implies $\pi_\phi\prec\gl$ (or even $\pi_\phi^\dd\prec\gl$).
\end{rem}

\subsection{Spectral gap and exponential mixing}\label{subsec:gap:mixing:general}\hfill\break
We establish an upper bound on the spectral radius of $\|\pi_\phi(f)\|_\text{op}$ for $\phi\in\mathbb{P}_+(\Lambda)$ and $f\in\BC[\Lambda]$ in terms of its Poincar\'e exponent and \text{regular spectral radius}, namely $\|\gl(f)\|_\text{op}$.

Given $f\in\BC[\Lambda]$ we denote $r(f)$ the minimal $r>0$ such that $\{f\neq0\}\subset B(e,r)$.

\vspace{5mm}
\noindent{\bf Theorem \ref{thm:general:1}}. 
{\it Given a positive cyclic representation, $\phi$,  on $\Lambda$ one has:
$$\|\pi_\phi(f)\|_\text{op}\le e^{\half\gd(\phi)r(f)}\|\gl(f)\|_\text{op}$$
for all $f\in \BC[\Lambda]$.}
\vspace{5mm}

In order to prove Theorem \ref{thm:general:1} we recall the following technical Lemma's proved in \cite{MR946351} (respectively p.101 and p.102):
\begin{lem}\label{lem:tech:haag:v1}
Given a unitary representation $(\pi,\mathcal{H})$, $\mathcal{H}_\infty\subset \mathcal{H}$ a dense subspace and $f\in\BC[\Lambda]$ one has:
$$\|\pi(f)\|_\text{op}= \sup_{v\in\mathcal{H}_\infty}\lim_n(\pi(f^**f^{*2n})v,v)^\frac{1}{4n}$$
\end{lem}
\begin{lem}\label{lem:tech:haag:v2}
For all $f\in\BC[\Lambda]$, $\lim_n\|(f^**f)^{*2n}\|_2^\frac{1}{4n}=\|\gl(f)\|_\text{op}$.
\end{lem}

\begin{proof}[Proof of Theorem \ref{thm:general:1}]
For all $f\in\BC[\Lambda]$ and $v=\pi_\phi(f){\bf1}_\phi$ one has 
$(\pi_\phi(g) v,v)=\ol{f}*\phi*f^\vee(g)$ with $f^\vee(g)=f(g^{-1})$ and, using Lemma \ref{lem:elem:control}: 
$$\gd(|(\pi_\phi v,v)|)\le \gd(|f|*\phi*|f^\vee|)\le\gd(\phi)$$
As $(\pi_\phi,{\bf1}_\phi)$ is a cyclic representation the subspace $\mathcal{H}_\infty\df\langle\pi_\phi(f){\bf1}_\phi:f\in\BC[\Lambda]\rangle\subset \mathcal{H}$ is dense.

It follows from Lemma \ref{lem:tech:haag:v1} that
\begin{align*}
\|\pi_\phi(f)\|_\text{op}=\sup_{v\in \mathcal{H}_\infty}\lim_n(\sum_\Lambda (f^**f)^{*2n}(g)(\pi_\phi(g)v,v))^\frac{1}{4n}
\end{align*}
for all $f\in\BC[\Lambda]$.

Let $S=\{f\neq0\}\subset B(e,r(f))$, then $S^{(n)}\df \{(f^**f)^{*2n}\neq0\}\subset B(e,4nr(f))$
and thus, for $s$ such that $s\gd=\gd(\phi)$:
\begin{align*}
&(\sum_\Lambda (f^**f)^{*2n}(g)(\pi_\phi(g)v,v))^2
=(\sum_{S^{(n)}} e^{\frac{s}{2}\gd d_\Lambda(g,e)}(f^**f)^{*2n}(g)e^{-\frac{s}{2}\gd d_\Lambda(g,e)}(\pi_\phi(g)v,v))^2\\
&\le e^{s\gd 4nr(f)}\|(f^**f)^{*2n}\|_2^2\sum_{S^{(n)}} e^{-s\gd d_\Lambda(g,e)}|(\pi_\phi(g)v,v)|^2\\
&\le e^{s\gd 4nr(f)}\|(f^**f)^{*2n}\|_2^2
(\sum_{S^{(n)}} e^{-(1+\frac{\e}{2})s\gd d_\Lambda(g,e)}|(\pi_\phi(g)v,v)|^{2+\e})^\frac{2}{2+\e}
|S^{(n)}|^\frac{\e}{2+\e}\\
&\prec e^{s\gd 4nr(f)}e^{\frac{\e}{2+\e}\gd 8nr(f)}\mathcal{P}(|(\pi_\phi v,v)|^{2+\e};s+\frac{\e s}{2})\|(f^**f)^{*2n}\|_2^2
\end{align*}
for any $\e>0$, where the last inequality uses the fact that: $|S^{(n)}|\le |B(e,4nr(f))|\le Ce^{8\gd nr(f)}$ for some constant $C$.

As a consequence of Lemma \ref{lem:elem:control} $\gd(|(\pi_\phi v,v)|^{2+\e})\le \gd(|(\pi_\phi v,v)|)$ and using Lemma \ref{lem:tech:haag:v2} :
\begin{align*}
\limsup_n(\sum_\Lambda (f^**f)^{*2n}(g)(\pi(g)v,v))^\frac{1}{4n}
&\le e^{\frac{\e}{2+\e}\gd r(f)}e^{\half s\gd r(f)}\|\gl(f)\|_\text{op}
\end{align*}
As $\e$ can be taken arbitrarily small the inequality follows. 
\end{proof}

\begin{rem}
The above inequality can be extended to all representation $\pi$ by considering $\pi\otimes\ol{\pi}$ where $\ol{\pi}$ stands for the contragredient dual representation of $\pi$ as every vector $v\otimes \ol{v}$ is positive, i.e $\phi=(\pi\otimes\ol{\pi} v\otimes \ol{v},v\otimes \ol{v})\in\mathbb{P}_+(\Lambda)$.
\end{rem}

Observe that given $\phi\in\mathbb{P}_+(\Lambda)$ in $\phi\in\ell^{p+\e}(\Lambda)$, i.e. $\phi\in \ell^{p'}(\Lambda)$ for all $p'>p$, Lemma \ref{lem:elem:control} implies that $\gd(\phi)\le \frac{1}{p'}\gd$ and using Theorem \ref{thm:general:1} it follows:
$$\|\pi_\phi(f)\|_\text{op}\le e^{\frac{p-1}{2p}\gd r(f)}\|\gl(f)\|_\text{op}$$

The above estimate can be improved assuming $p\ge2$ as follows:

\vspace{5mm}
\noindent{\bf Theorem \ref{thm:general:2}}. 
{\it 
Let $(\pi,\mathcal{H}_\pi)$ be a strongly $L^{p+\e}$ unitary representation with $2\le p\neq \infty$, then:
$$\|\pi(f)\|_\text{op}\le e^{\frac{p-2}{2p}\gd r(f)}\|\gl(f)\|_\text{op}$$
for all $f\in\BC[\Lambda]$.}
\vspace{5mm}

It would follows from the same arguments (or \cite[Thm. 1]{MR946351})  that $\|\pi_\phi(f)\|_\text{op}\le \|\gl(f)\|_\text{op}$ if $\phi\in\ell^{p+\e}(\Lambda)\cap\mathbb{P}_+(\Lambda)$ for any $1\le p<2$. One the other hand  Shalom's trick \cite[Lem. 2.3]{MR1792293} shows that this inequality is tight.

\begin{proof}
We use the same notations as in Theorem \ref{thm:general:1} proof.
Given ${p'}>p$ we consider $\mathcal{H}_\infty\df\mathcal{D}_{p'}\subset \mathcal{H}_\pi$ a dense subspace such that $(\pi(.) v,w)\in\ell^{p'}(\Lambda)$ for all $v,w\in \mathcal{D}_{p'}$.

Note that for all $\kappa>0$, there exists $n_0\in\BN$ such that
$$|S^{(n)}|\le |B(e,4nr(f))|\prec e^{(\gd+\kappa)4nr(f)}$$
for all $n\ge n_0$.
For $v\in \mathcal{H}_\infty$ and $\e$ such that $p'<p+\e$ one has: 
\begin{align*}
&(\sum_{g\in\Ga} (f^**f)^{*2n}(g)(\pi_\phi(g)v,v))^2
\le\|(f^**f)^{*2n}\|_2^2\sum_{S^{(n)}}|(\pi_\phi(g)v,v)|^2\\
&\le \|(f^**f)^{*2n}\|_2^2 |S^{(n)}|^\frac{p+\e-2}{p+\e}(\sum_{S^{(n)}}|(\pi_\phi(g)v,v)|^{p+\e})^\frac{2}{p+\e}\\
&\prec (\sum_{g\in\Ga} |(\pi_\phi(g)v,v)|^{p+\e})^\frac{2}{p+\e}\|(f^**f)^{*2n}\|_2^2e^{\frac{p+\e-2}{p+\e}(\gd+\kappa) 4nr(f)}
\end{align*}
and thus
\begin{align*}
\|\pi_\phi(f)\|_\text{op}&=\sup_{v\in \mathcal{H}_\infty}\lim_n(\sum_\Ga (f^**f)^{*2n}(g)(\pi_\phi(g)v,v))^\frac{1}{4n}\\
&\le e^{\frac{p+\e-2}{2p+2\e}(\gd+\kappa) r(f)}\|\gl(f)\|_\text{op}
\end{align*}
As $\e,\kappa>0$ can be taken arbitrarily small:
$$\|\pi_\phi(f)\|_\text{op}\le e^{\frac{p-2}{2p}\gd r(f)}\|\gl(f)\|_\text{op}=e^{(s-\half)\gd r(f)}\|\gl(f)\|_\text{op}$$
with $s=\frac{p-1}{p}$.
\end{proof}

\begin{cor}
Let $(\Lambda,d_\Lambda)$ be a discrete metric group with the radial rapid decay property and $(\gb_r)_r$ be the uniform probabilities over the balls of radius $r$.
Then there exists $m\in\BN$, for all $\e>0$, there exists $C>0$, for all $\phi\in\mathbb{P}_+(\Lambda)$ :
$$\|\pi_\phi(\gb_r)\|_\text{op}\le Cr^m e^{\half (\gd(\phi)+\e-\gd)r}$$
\end{cor}
In other words $\pi_\phi$ is exponentially mixing whenever $\phi\in\mathbb{P}_+(\Lambda)$ has a critical gap, i.e $\gd(\phi)<\gd$.
\begin{proof}

Let $r_0\ge0$ such that $|B(e,r)|\prec e^{(\gd+\e)r}$ for all $r\ge r_0$.
Using Theorem \ref{thm:general:1}:
\begin{align*}
\|\pi_\phi(\gb_r)\|_\text{op}&\le e^{\half(\gd(\phi)) r}\|\gl(\gb_r)\|_\text{op}\\
&\le_\text{RRD} Cr^me^{\half(\gd(\phi)) r}\|\gb_r\|_2\\
&\le C'r^me^{\half(\gd(\phi)+\e-\gd) r}
\end{align*}
\end{proof}

A stronger result can be deduced assuming $\Lambda$ has the RD property:
\begin{cor}\label{cor:relative:mx}
Let $(\Lambda,d_\Lambda)$ be a discrete metric group with the rapid decay property, $S\subset \Lambda$ a subset of critical exponent $\gd_S$ and $(\gb_r(S))_r$ the uniform probabilities on $S\cap B(e,r)$.
There $m\in\BN$, for all $\e>0$, there exists a constants $C>0$, for all $\phi\in\mathbb{P}_+(\Lambda)$ :
$$\|\pi_\phi(\gb_r(S))\|_\text{op}\le Cr^m e^{\half (\gd(\phi)+\e-\gd_S)r}$$
\end{cor}
\begin{exam}
Let $\Lambda$ with RD, $\phi\in \mathbb{P}_+(\Lambda)$ and $H\subset \Lambda$ such that $\gd(\phi)<\gd_{H}<\gd_\Lambda$ then $\pi_\phi|_{H}$ is exponentially mixing.
In particular if $H_1, H_2\subset \Lambda$ and  $\gd_{H_1}<\gd_{H_2}<\gd_\Lambda$ then $(\gl|_{H_2},\ell^2(\Lambda/H_1))$ is exponentially mixing.
\end{exam}

\begin{proof}[Corollary \ref{cor:relative:mx} proof]
Let $r_0\ge0$ such that $|B(e,r)\cap S|\prec e^{(\gd_S+\e)r}$ for all $r\ge r_0$.
Using Theorem \ref{thm:general:1}:
\begin{align*}
\|\pi_\phi(\gb_r)\|_\text{op}&\le e^{\half(\gd(\phi)) r}\|\gl(\gb_r(S))\|_\text{op}\\
&\le_\text{RD} Cr^me^{\half(\gd(\phi)) r}\|\gb_r(S)\|_2\\
&\le C'r^me^{\half(\gd(\phi)+\e-\gd_S) r}
\end{align*}
\end{proof}

\subsection{Quantitative property (T) for hyperbolic groups}\label{subsec:quant:prop:T}\hfill\break
This subsection is dedicated to the proof of Theorem \ref{thm:quant:T}.

Let $S\subset\Lambda$ be a finite subset of a discrete group $\Lambda$ and $\e>0$. 
A isometric representation on a Hilbert space, $(\pi,\mathcal{H}_\pi)$, of $\Lambda$ has a $(S,\e)$-almost invariant vector if there exists $v\in \mathcal{H}_\pi$ with $\|v\|=1$ such that $\|\pi(g)v-v\|\le \e$ for all $g\in S$. 
More generally a representation $\pi$ has almost invariant vectors if $S\subset\Lambda$ finite and $\e>0$ can be chosen arbitrarily.
\begin{defn}
A discrete $\Lambda$ has the Kazhdan property (T) if every unitary representation with almost invariant vectors has a non-zero invariant vector.
\end{defn}
It follows from \cite{MR1995802} that generic Gromov hyperbolic groups have the Kazhdan property (T) (see \cite{MR2415834} for more about property (T)).

\begin{prop}\label{prop:T:crit}
Let $\Ga$ be a Gromov hyperbolic group with the Kazhdan property (T).
There exists $0\le\gd_0<\gd$ such that for all positive cyclic representation without invariant vector, $\phi$, the critical exponent of $\phi$ satisfies $\gd(\phi)\le \gd_0$.
\end{prop}
\begin{proof}
It follows from \cite[p.33]{MR2415834} that if $\Ga$ has the Kazhdan property (T) there exist $S\subset \Ga$ finite and $\e>0$ such that any representation with a $(S,\e)$-almost invariant has a non-zero invariant vector.

Assume one can find $(\phi_n\simeq[\pi_{\phi_n},{\bf1}_{\phi_n}])_n$ a sequence of positive cyclic representations without $(S,\e)$-almost invariant vectors such that $\gd(\phi_n)\rightarrow\gd$.
It follows from Lemma \ref{lem:seq:weak:cont} and Theorem \ref{thm:amenable:gene} that  
$\pi_{\phi_n}\xrightarrow{n} 1$ for the Fell topology.
In particular there exists $N\in\BN$, for all $n\ge N$, $\pi_{\phi_n}$ has a $(S,\e)$-almost invariant vector which is a contradiction.
Therefore one can find $0\le\gd_0<\gd$ such that whenever $\phi\in\mathbb{P}_+(\Ga)$ does not have a $(S,\e)$-almost invariant vectors $\gd(\phi)\le\gd_0$.
\end{proof}

\begin{rem}
If we consider the spectrum of critical values:
$$\text{Spec}_+(\Ga,d_\Ga)\df\{\gd(\pi):[\pi,v]\in\mathbb{P}_+(\Ga)\}$$
Proposition \ref{prop:T:crit}  implies that $1$ is isolated in $\text{Spec}(\Ga,d_\Ga)$.
\end{rem}

\begin{proof}[Proof of Theorem \ref{thm:quant:T}]
Given a probability $\mu\in\tprob(\Ga)$ on $\Ga$:
\begin{align*}
\|\pi(\mu)v\|^4&=(\sum \mu^**\mu(g)(\pi(g)v,v))^2\\
&\le \sum \mu^**\mu(g)|(\pi(g)v,v)|^2\\
&=\sum_g \mu^**\mu(g)(\pi\otimes\ol{\pi}(g)v\otimes \ol{v},v\otimes \ol{v})\\
&=\|\pi\otimes\ol{\pi}(\mu)v\otimes \ol{v}\|^2\le \|\pi\otimes\ol{\pi}|_{W_v}(\mu)\|_\text{op}^2\|v\|^4
\end{align*}
with $W_v=\text{Span}\{\pi\otimes\ol{\pi}(g)v\otimes \ol{v}:g\in\Ga\}\subset \mathcal{H}\otimes \ol{\mathcal{H}}$.
As $[\pi\otimes\ol{\pi}|_{W_v},v\otimes\ol{v}]$ is a positive cyclic representation and does not have invariant vectors since $\pi$ is weakly mixing, it follows from Proposition \ref{prop:T:crit} that $\gd[\pi\otimes\ol{\pi}|_{W_v},v\otimes\ol{v}]\le \gd_0$ for some $0\le\gd_0<\gd$ that only depends on $(\Ga,d_\Ga)$.
On the other hand Theorem \ref{thm:general:1} combined with rapid decay property for hyperbolic groups \cite{MR943303} lead to
$$\|\text{Avr}_r(\pi\otimes\pi|_{W_v})\|_\text{op}\le e^{\half\gd_0 r}\|\text{Avr}_r(\gl)\|_\text{op}$$
which implies 
$$\frac{\|\text{Avr}_r(\pi)v\|}{\|v\|}\le \|\text{Avr}_r(\pi\otimes\pi|_{W_v})\|_\text{op}^\half\le e^{\frac{1}{4}\gd_0 r}\|\text{Avr}_r(\gl)\|_\text{op}^\half$$
It follows that 
$$\|\text{Avr}_r(\pi)\|_\text{op}\le e^{\frac{1}{4}\gd_0 r}\|\text{Avr}_r(\gl)\|_\text{op}^\half\le Cr(\mu)^me^{\frac{1}{4}(\gd_0-\gd) r(\mu)}$$ for $C>0$ and $m\in\BN$ and thus
$h(\pi)\ge\frac{(\gd-\gd_0)}{4}>0$.
\end{proof}

\subsection{Spectral gap and boundary representations}\label{subsec:last}\hfill\break
We investigate the optimal mixing rate of $\phi\in\mathbb{P}_+(\Ga)$ and its boundary representations $\pi_\phi^\dd$.

We denote $\mu$ a quasi-invariant probability measure on $\ol{\Ga}$ and $E$ a $(\Ga,s,\mu)$-conditional expectation as in Definition \ref{def:gamma:s:cond}.
Recall that $f\in L^2(\mu)\mapsto (E(f),f)_{L^2(\mu)}$ defines a $\pi_s$-invariant quadratic form on $L^2(\mu)$.\\

\begin{lem}\label{lem:gene:cocycle:v} 
Let $\Lambda$ be a discrete group, $(Z,\nu)$ a probability space such that $\Lambda$ acts by measure class preserving transformations and $f\in \BR_+[\Lambda]$ a positive finitely supported function on $\Lambda$.
Let $B:\Lambda\times Z\rightarrow(\BR,+)$ be a measurable additive cocycle and $\pi_B$ the linear representation on $L^2(\nu)$ given by $\pi_B(g)\psi=e^{-B(g;\bullet)}g^*\psi$ for $\psi\in L^2(\nu)$ and $g\in\Lambda$.
Then the bounded operator $\pi_B(f)$ on $L^2(\nu)$ satisfies 
$$\|\pi_B(f)|L^2(\nu)\|_\text{op}\le \sqrt{\|\pi_B(f){\bf1}|L^\infty(\nu)\|.\|\pi_B^*(f^\vee){\bf1}|L^\infty(\nu)\|}$$
where $f^\vee(g)=f(g^{-1})$ for $g\in\Lambda$.
\end{lem}
\begin{proof}
For $v,w\in L_+^2(\nu)$ one has:
\begin{align*}
&|(\pi(f)v,w)_{L^2(\nu)}|^2
=|\int_{\ol{\Ga}}\sum_{g\in\Ga} f(g)e^{-B(g;z)}v(g^{-1}z)w(z)d\nu(z)|^2\\
&\le \int_{\ol{\Ga}\times\Ga}f(g)e^{-B(g;z)}|v|^2(g^{-1}z)d\nu(z)dg
.\int_{\ol{\Ga}\times\Ga}f(g)e^{-B(g;z)}|w|^2(z)d\nu(z)dg
\end{align*}
\begin{align*}
&= (|v|^2,\pi_B^*(f^\vee){\bf1})_{L^2(\nu)}
(\pi_B(f){\bf1},|w|^2)_{L^2(\nu)}\\
&=\|\pi_B(f){\bf1}|L^\infty(\nu)\|
\|\pi_B^*(f^\vee){\bf1}|L^\infty(\nu)\|
\|v|L^2(\nu)\|^2\|w|L^2(\nu)\|^2
\end{align*}
which concludes the proof.
\end{proof}

\begin{lem}\label{lem:gener:cond}
Let $\mu$ be a quasi-invariant probability on $\ol{\Ga}$, $s\in[\half,1]$, $E$ a $(\Ga,s,\mu)$-conditional expectation and $f\in \BR_+[\Ga]$ with $f^\vee=f$.
The bounded operator $\pi_{s}(f)$ on $L^2(\mu)$ satisfies
$$\|\pi_{s}(f)\|_\text{op}\le \|\pi_{s}(f){\bf1}|L^\infty(\mu)\|$$
\end{lem}
\begin{proof}
As 
$E(\pi_s(g)\psi)=\pi_{s}^*(g)E(\psi)$
for $\psi\in L^2(\mu)$ and $g\in\Ga$ it follows $E(\pi_{s}(f){\bf1})=\pi_{s}^*(f)E({\bf1})=\pi_{s}^*(f){\bf1}$
and thus 
$$\|\pi_s^*(f){\bf1}\|_\infty= \|E(\pi_s(f){\bf1})\|_\infty \le \|\pi_s(f){\bf1}\|_\infty$$
as $E$ is a contraction on $L^\infty(\nu)$.
The bound follows from the above estimate and Lemma \ref{lem:gene:cocycle:v} with $(Z,\nu)=(\ol{\Ga},\mu)$ and $B=sb$.
\end{proof}

Assuming $E$ is positive definite on $L^2(\mu)$, we denote $\mathcal{H}(\mu,E)$ the Hilbert completion of $(L^2(\mu),E)$.

\begin{cor}\label{cor:control:general:l2}
Let $\mu$ be a quasi-invariant probability on $\ol{\Ga}$, $s\in[\half,1]$, $E$ a positive definite $(\Ga,s,\mu)$-conditional expectation and $f\in \BR_+[\Ga]$ with $f^\vee=f$.
The bounded operator $\pi_s(f)$ on $\mathcal{H}(\mu,E)$ satisfies
$$\|\pi_s(f)|\mathcal{H}(\mu,E)\|_\text{op}\le \|\pi_s(f){\bf1}|L^\infty(\mu)\|$$
\end{cor}
\begin{proof}

As $\pi_s$ is unitary on $\mathcal{H}(\mu,E)$ and $f^\vee=f$, the operator $\pi_s(f)$ is self-adjoint on $\mathcal{H}(\mu,E)$.

In particular $\pi_s(f)^*\pi_s(f)|_{L^2(\mu)}=\pi_s(f)^2$ and 
\begin{align*}
\|\pi_s(f)|\mathcal{H}(\mu,E)\|_\text{op}&\le_\text{Lem. \ref{lem:tech:op:incl}} \|\pi_s(f)^*\pi_s(f)|L^2(\mu)\|_\text{op}^\half\\
&\le \|\pi_s(f)|L^2(\mu_{\gs_2})\|_\text{op}\\
&\le_\text{lem. \ref{lem:gener:cond}} \|\pi_s(f){\bf1}|L^\infty(\mu)\|
\end{align*} 
\end{proof}

The following proposition gives a sharp norm estimates for certain radial averages on $(\pi_s,\mathcal{H}(\mu,E))$:

\begin{prop}\label{prop:red:esti}
Let $\mu$ be a quasi-invariant probability on $\dd\Ga$, $E$ a positive definite $(\Ga,s,\mu)$-conditional expectation.
Then 

\[   
\|\sum_{g: d_\Ga(g,e)\le L}\pi_s(g)|\mathcal{H}(\mu,E)\|_\text{op}\asymp
     \begin{cases}
       \frac{1}{(2s-1)}e^{s\gd L}& \text{for $s>\half$}\\
       e^{s\gd L}L&\text{for $s=\half$} 
     \end{cases}
\]
for $L\ge R$.

\end{prop}

\begin{proof}
Observe that
\begin{align*}
\inf_{\dd \Ga}\pi_s(f){\bf1}&\le (E(\pi_s(f){\bf1}),\pi_s(f){\bf1})_{L^2(\mu)}^\half
\le\|\pi_s(f)|\mathcal{H}(\mu,E)\|_\text{op}\\
&\le_\text{Cor. \ref{cor:control:general:l2}}\|\pi_s(f){\bf1}|L^\infty(\mu)\|
\end{align*}
and thus 
$\|\pi_s(f)|\mathcal{H}(\mu,E)\|_\text{op}\asymp\|\pi_s(f){\bf1}|L^\infty(\mu)\|$ whenever 
$\|\pi_s(f){\bf1}|L^\infty(\mu)\|\prec\inf_{\dd \Ga}\pi_s(f){\bf1}$.

We conclude using Lemma \ref{cor:ball:spec:tech}.
\end{proof}

\begin{cor}\label{cor:bound:rep:mixing}
Let $\mu$ be a quasi-invariant probability on $\dd{\Ga}$, $s\in[\half,1]$ and $E$ a positive definite $(\Ga,s,\mu)$-conditional expectation.
Then
$$\lim_{r\rightarrow+\infty}-\frac{1}{r}\log\|\frac{1}{|B(e,r)|}\sum_{g: d_\Ga(g,e)\le r}\pi_s(g)|\mathcal{H}(\mu,E)\|_\text{op}=\gd-s\gd$$
\end{cor}
\begin{proof}
It follows from Proposition \ref{prop:red:esti} and pure sphericality of Gromov hyperbolic groups \cite[Thm. 7.2]{MR1214072} that
\[   
\|\frac{1}{|B(e,r)|}\sum_{g: d_\Ga(g,e)\le L}\pi_s(g)|\mathcal{H}(\mu,E)\|_\text{op}\asymp
     \begin{cases}
       \frac{1}{(2s-1)}e^{-(1-s)\gd L}& \text{for $s>\half$}\\
       e^{-\half\gd L}L&\text{for $s=\half$} 
     \end{cases}
\]

and thus 
$$\lim_{r\rightarrow+\infty}-\frac{1}{r}\log\|\frac{1}{|B(e,r)|}\sum_{g: d_\Ga(g,e)\le r}\pi_s(g)|\mathcal{H}(\nu,E)\|_\text{op}=(1-s)\gd$$

\end{proof}

In particular the boundary representations, $\pi_\phi^\dd$, are exponentially mixing with entropy equal to $\gd-\max\{\gd(\phi),\half\gd\}$.

\begin{rem}
Given $\phi,\phi'\in\mathbb{P}_+(\Ga)$ with $\gd(\phi),\gd(\phi')\ge\half\gd$ and their respective boundary representations $\pi_\phi^\dd$ and $\pi_{\phi'}^\dd$, a necessary condition for these representations to be unitary equivalent is that $\gd(\phi)=\gd(\phi')$.
\end{rem}

\begin{cor}\label{cor:pos:rep:spectrapgap}
Given $\phi\in\mathbb{P}_+(\Ga)$ one has:
$$\lim_{r\rightarrow+\infty}-\frac{1}{r}\log\|\frac{1}{|B(e,r)|}\sum_{g: d_\Ga(g,e)\le r}\pi_\phi(g)\|_\text{op}=\gd-\max\{\gd(\phi),\half\gd\}$$
\end{cor}

\begin{proof}
Let $\gs\gd>\max\{\gd(\phi),\half\gd\}$ and $J_\gs$ the intertwiner considered in Subsection \ref{subsec:gns:twist} with $\go_\gs(g)=e^{-\gs\gd d_\Ga(g,e)}$ for $g\in \Ga$ (see also Subsection \ref{subsec:twist:poincare}).
One has:
\begin{align*}
\|\frac{1}{|B(e,r)|}\sum_{g: d_\Ga(g,e)\le r}\pi_\phi(g)&|\mathcal{H}_\phi\|_\text{op}
=\|J_\gs(\frac{1}{|B(e,r)|}\sum_{g: d_\Ga(g,e)\le r}\pi_\gs(g))J_\gs^{-1}|\mathcal{H}_\phi\|_\text{op}\\
&=\|\frac{1}{|B(e,r)|}\sum_{g: d_\Ga(g,e)\le r}\pi_\gs(g)|\mathcal{H}(\mu_\gs,E_\gs)\|_\text{op}\\
&\le_{Cor. \ref{cor:control:general:l2}} \|\frac{1}{|B(e,r)|}\sum_{g: d_\Ga(g,e)\le r}\pi_\gs(g){\bf1}\|_\infty\\
&\prec_\text{Cor. \ref{cor:ball:spec:tech}} e^{-(1-\gs)\gd r}
\end{align*}

It follows that 
$$\gd-\gs\gd\le\liminf_r-\frac{1}{r}\log\|\frac{1}{|B(e,r)|}\sum_{g: d_\Ga(g,e)\le r}\pi_\phi(g)|\mathcal{H}_\phi\|_\text{op}$$
for all $\gs>\max\{s(\phi),\half\}$.
On the other hand Lemma \ref{lem:seq:weak:cont} implies that $\pi_\phi^\dd\prec\pi_\phi$ and thus
$$\limsup_r-\frac{1}{r}\log\|\pi_\phi(\gb_r)\|\le\limsup_r-\frac{1}{r}\log\|\pi_\phi^\dd(\gb_r)\|\le_\text{lem \ref{cor:bound:rep:mixing}}\gd-\max\{\gd(\phi),\half\gd\}$$

\end{proof}

In particular:
\begin{cor}\label{cor:gap:eq:almost}
Given $\phi\in\mathbb{P}_+(\Ga)$, $\gd(\phi)<\gd$ if and only if $1\nprec\phi$.
\end{cor}
\begin{proof}
If $\gd(\phi)=\gd$ it follows from Theorem \ref{thm:amenable:gene} that $1\prec\phi$.
If $\gd(\phi)<\gd$, then $\pi_\phi$ has a spectral gap due to Corollary \ref{cor:pos:rep:spectrapgap} and thus $1\nprec\phi$.
\end{proof}

We conclude with the following observation:

\begin{cor}\label{cor:last:of:the:last}
Let $\phi_1,\phi_2\in\mathbb{P}_+(\Ga)$ with $\gd(\phi_1),\gd(\phi_2)\ge\half\gd$.
Then  $\pi_{\phi_1}\prec\pi_{\phi_2}$ implies $\gd({\phi_1})\le\gd({\phi_2})$.
In particular $\pi_{\phi_1}\simeq \pi_{\phi_2}$ implies $\gd({\phi_1})=\gd({\phi_2})$.
\end{cor}
\begin{proof}
As $\pi_{\phi_1}\prec \pi_{\phi_2}$ one has $\|\pi_{\phi_1}(f)\|\le \|\pi_{\phi_2}(f)\|$ and thus
\begin{align*}
\gd-\max\{\gd(\phi_2),\half\gd\}&= \lim_{r\rightarrow+\infty}-\frac{1}{r}\log\|\frac{1}{|B(e,r)|}\sum_{g: d_\Ga(g,e)\le r}\pi_{\phi_2}(g)\|\\
&\le\lim_{r\rightarrow+\infty}-\frac{1}{r}\log\|\frac{1}{|B(e,r)|}\sum_{g: d_\Ga(g,e)\le r}\pi_{\phi_1}(g)\|\\
&= \gd-\max\{\gd(\phi_1),\half\gd\}
\end{align*}

\end{proof}

\appendix
\section{Norm transfer estimate}
\begin{lem}\label{lem:tech:op:incl}
Let $(B,\|*\|_B)\subset (\mathcal{H},\|*\|_\mathcal{H})$ be a continuous embedding of a Banach space $B$ into a Hilbert space, $\mathcal{H}$, with dense range and $T$ a densely defined operator of domain $\mathcal{D}(T)$ on $\mathcal{H}$ such that $B\subset \mathcal{D}(T)$, $TB\subset \mathcal{D}(T^*)$ and $T^*T$ restrict as a bounded operator on $B$, where $T^*$ stands for the adjoint in $\mathcal{H}$ .
Then $T$ is bounded on $\mathcal{H}$ and $$\|T|\mathcal{H}\|_\text{op}\le\|T^*T|B\|_\text{op}^\half$$
\end{lem}

\begin{proof}
Identifying $B$ with a dense subspace of $\mathcal{H}$:
\begin{align*}
\|T|\mathcal{H}\|_\text{op}
=\sup_{v\in B\subset \mathcal{D}(T)}\frac{\|Tv\|_\mathcal{H}}{\|v\|_\mathcal{H}}
\end{align*}
On the other hand, as $TB\subset \mathcal{D}(T^*)$:
$$\|Tv\|_\mathcal{H}=(T^*Tv,v)^\half\le \|T^*Tv\|_\mathcal{H}^\half\|v\|_\mathcal{H}^\half$$
and by induction, as $T^*T$ maps $B$ to itself:
$$\|Tv\|_\mathcal{H}\le \|(T^*T)^{2^{n}}v\|_\mathcal{H}^{2^{-n-1}}\|v\|_\mathcal{H}^{1-2^{-n-1}}$$

It follows that:
$$\frac{\|Tv\|_\mathcal{H}}{\|v\|_\mathcal{H}}\le \liminf_n\|(T^*T)^{2n}v\|_\mathcal{H}^\frac{1}{4n}$$

As $T^*T$ induces a bounded operator on $B$:
\begin{align*}
\|(T^*T)^{2n}v\|_\mathcal{H}&\le c(B,\mathcal{H})\|(T^*T)^{2n}v\|_B\\
&\le c(B,\mathcal{H})\|(T^*T)^{2n}|B\|_\text{op}\|v\|_B\\
&\le c(B,\mathcal{H})\|T^*T|B\|_\text{op}^{2n}\|v\|_B
\end{align*}
for all $v\in B$, where $c(B,\mathcal{H})$ denote the norm of the embedding $B\subset \mathcal{H}$. 

It follows that 
$$\liminf_n\|(T^*T)^{2n}v\|_\mathcal{H}^\frac{1}{4n}\le \|T^*T|B\|_\text{op}^\half. \liminf_n\|v\|_B^\frac{1}{4n}$$
and thus 
$\|T|\mathcal{H}\|_\text{op}\le \|T^*T|B\|_\text{op}^\half$

\end{proof}

\section{Weak operator topology}\label{appendix:WOT}
Given two normed spaces $(X,N_X),(Y,N_Y)$ a sequence $(T_n)_n$ of uniformly bounded operators converges for the weak operator topology (WOT) to $T$ if $T$ is bounded and  for all $x\in X$ and $\vp\in Y'$ one has $\vp(T_n(x))\rightarrow \vp(T_n(x))$.

\begin{thm}
Let $X,Y$ be two normed spaces.
Assume $Y$ is a Banach reflexive space.
The space of bounded operators with norms at most $1$ between $X$ and $Y$, $\text{B}_{\le 1}(X,Y)$, endowed with the weak operator topology is compact.
\end{thm}
\begin{proof}
The proof almost follows the same lines as the Hilbert case (see \cite[Thm. 5.1.3]{MR1468229} ).
Using reflexivity together with Banach-Alaoglu and Tychonoff theorems (\cite{MR2378491}) we deduce that the space $K=(Y_{\le 1},\gs(Y,Y'))^{X_{\le 1}}$ is compact.
Moreover the map $$A:(\text{B}_{\le 1}(X,Y),WOT)\rightarrow K;\quad T\mapsto [A(T):x\mapsto T(x)]$$
is a homeomorphism onto its image $A((\text{B}_{\le 1}(X,Y),WOT))=K'\subset K$ from the definition of the topologies involved.
To prove compactness $K'$ it is enough to prove that it is closed.
Using an approximation argument, every elements in $\ol{K'}$ define a linear map between $X$ and $Y$.
Given $T\in\ol{K'}$, $(T_i)_i$  a sequence in $K'$ converging to $T$, $x\in X_{\le 1}$ and $\vp\in Y_{\le 1}'$ such that $\|T(x)\|=\vp(T(x))$, thanks to Hahn-Banach theorem, one has
$\|T(x)\|=\vp(T(x))=\liminf_i\vp(T_i(x))\le \liminf\|T_i\|\le 1$.
In other words $T\in \text{B}_{\le 1}(X,Y')$.
\end{proof}

\begin{lem}\label{lem:app:wot:comp}
Let $X,Y$ be two normed spaces.
The map $$B:(\text{B}_{\le 1}(X,Y'),WOT)\times (X,N_X)\rightarrow (Y',\gs(Y',Y));(T,x)\mapsto T(x)$$
is continuous.
\end{lem}

\begin{proof}
Let $(T_i)_i$ and $(x_j)_j$ be two converging net respectively to $T$ and $x$ in $\text{B}_{\le 1}(X,Y)$ and $X$.
Given $y\in Y$ one can find $i_0$ and $j_0$ such that $|T_i(x)(y)-T(x)(y)|\le \e$ and $N_X(x_j-x)\le \e$ for respectively $i\ge i_0$ and $j\ge j_0$.
It follows follows
\begin{align*}
|T_i(x_j)(y)-T(x)(y)|&\le |T_i(x_j)(y)-T_i(x)(y)|+|T_i(x)(y)-T(x)(y)|\\
&\le N_X(x_j-x)+\e\le 2\e
\end{align*}
for $i\ge i_0$ and $j\ge j_0$.
\end{proof}


\nocite{*}
\bibliographystyle{plain}
\bibliography{bib-poincare-and-co}
\end{document}